\documentclass[a4paper,reqno,10pt]{amsart}

\usepackage{epsf,graphicx,epsfig}
\usepackage{amscd}
\usepackage{amsmath,latexsym,amssymb,amsthm}
\usepackage[nospace,noadjust]{cite}
\usepackage{textcomp}
\usepackage{amsfonts}
\usepackage{setspace,cite}
\usepackage{lscape,fancyhdr,fancybox}
\usepackage{stmaryrd}
\usepackage{mathrsfs}
\usepackage[all,cmtip]{xy}
\usepackage{tikz}
\usepackage{cancel}
\usetikzlibrary{shapes,arrows,decorations.markings}
\usepackage{graphicx}
\usepackage{tikz-cd}
\usepackage{color}
\usepackage{url}
\usepackage{enumerate}
\usepackage[mathscr]{euscript}

  \theoremstyle{plain}

\swapnumbers
    \newtheorem{thm}{Theorem}[section]
    \newtheorem{prop}[thm]{Proposition}
     
   \newtheorem{lemma}[thm]{Lemma}

    \newtheorem{subsec}[thm]{}
\theoremstyle{definition}
    \newtheorem{defn}[thm]{Definition}
        \newtheorem{remark}[thm]{Remark}
    \newtheorem{exam}[thm]{Example}

\theoremstyle{remark}

\title{}
\author{}
\date{}
\usepackage{amssymb}

\usepackage{hyperref}
\hypersetup{
	colorlinks,
	citecolor=blue,
	filecolor=black,
	linkcolor=blue,
	urlcolor=black
}

\begin{document}

\title[Deformation maps  in proto-twilled Leibniz algebras]{Deformation maps  in proto-twilled Leibniz algebras}

\author{Apurba Das}
\address{Department of Mathematics,
Indian Institute of Technology, Kharagpur 721302, West Bengal, India}
\email{apurbadas348@gmail.com, apurbadas348@maths.iitkgp.ac.in}

\author{Suman Majhi}
\address{Department of Mathematics,
Indian Institute of Technology, Kharagpur 721302, West Bengal, India}
\email{majhisuman693@gmail.com}

\author{Ramkrishna Mandal}
\address{Department of Mathematics, Indian Institute of Technology, Kharagpur 721302, West Bengal, India}
\email{ramkrishnamandal430@gmail.com}

\begin{abstract}
This paper aims to find a unified approach to studying the cohomology theories of various operators on Leibniz algebras. We first introduce deformation maps in a proto-twilled Leibniz algebra to do this. Such maps generalize various well-known operators (such as homomorphisms, derivations, crossed homomorphisms, Rota-Baxter operators, modified Rota-Baxter operators, twisted Rota-Baxter operators, Reynolds operators etc) defined on Leibniz algebras and embedding tensors on Lie algebras. We define the cohomology of a deformation map unifying the existing cohomologies of all the operators mentioned above. Then we construct a curved $L_\infty$-algebra whose Maurer-Cartan elements are precisely deformation maps in a given proto-twilled Leibniz algebra. In particular, we get the Maurer-Cartan characterizations of modified Rota-Baxter operators, twisted Rota-Baxter operators and Reynolds operators on a Leibniz algebra. Finally, given a proto-twilled Leibniz algebra and a deformation map $r$, we construct two governing $L_\infty$-algebras, the first one controls the deformations of the operator $r$ while the second one controls the simultaneous deformations of both the proto-twilled Leibniz algebra and the operator $r$.
\end{abstract}

\maketitle



\medskip

\begin{center}

\end{center}

 



\thispagestyle{empty}

\tableofcontents


\medskip

\section{Introduction}
Leibniz algebras are a noncommutative analogue of Lie algebras. Bloh \cite{bloh} first introduced these algebras under the name of $D$-algebras and later rediscovered them by Loday \cite{loday} who called them Leibniz algebras. In the same paper, Loday introduces the homology theory of Leibniz algebras as the noncommutative generalization of Lie algebra homology. Subsequently, Loday and Pirashvili \cite{loday-pira} defined the cohomology theory of Leibniz algebras with coefficients in a representation. This cohomology can also be regarded as the noncommutative analogue of the Chevalley-Eilenberg cohomology of Lie algebras. In \cite{bala} Balavoine defined a graded Lie bracket (known as {\em Balavoine bracket}) whose Maurer-Cartan elements correspond to Leibniz algebra structures on a given vector space. This characterization allows him to describe Loday-Pirashvili's coboundary operator of a Leibniz algebra in terms of the Balavoine bracket and study formal deformations of a Leibniz algebra. Various studies on Leibniz algebras have been conducted in connection with mathematics and mathematical physics in the last twenty years. In particular, homotopy theory \cite{aamar,khuda}, integration \cite{covez,bor} and deformation quantization \cite{dherin} of Leibniz algebras are well-studied. As the underlying algebraic structures of embedding tensors, Leibniz algebras also have close connections with higher gauge theories \cite{sheng-embed,kotov}.

\medskip

Given an algebraic structure (e.g. associative, Lie or Leibniz), one may define various interesting operators on it. Algebra homomorphisms, derivations and crossed homomorphisms are the most well-known operators that are often useful for understanding the underlying algebraic structure. In the last few years, Rota-Baxter operators and their cousins (relative Rota-Baxter operators, Rota-Baxter operators with weight, modified Rota-Baxter operators, twisted Rota-Baxter operators, Reynolds operators and averaging operators etc.) were extensively studied on Lie algebras and associative algebras. They have close connections with combinatorics, splitting of algebras, bialgebra theory and integrable systems. It is important to remark that a (relative) Rota-Baxter operator on a Lie algebra is intimately related to the solutions of the classical Yang-Baxter equation \cite{kuper} while a Rota-Baxter operator on an associative algebra appears in the renormalization of quantum field theory \cite{conn}. In \cite{tang} the authors first introduced the cohomology and deformations of a (relative) Rota-Baxter operator defined on a Lie algebra. Subsequently, the cohomology and deformations of a (relative) Rota-Baxter Lie algebra (i.e. Lie algebra endowed with a (relative) Rota-Baxter operator) was also studied \cite{lazarev}. The cohomological study of other cousins of Rota-Baxter operators on Lie algebras and associative algebras was also investigated in several papers. See \cite{sheng} and the references therein for more details. To unify the cohomologies and deformations of various operators mentioned above, recently Jiang, Sheng and Tang \cite{jiang-sheng-tang} considered quasi-twilled Lie algebras and two types of deformation maps on it. A deformation map of type I unifies Lie algebra homomorphisms, derivations, crossed homomorphisms and modified Rota-Baxter operators while a deformation map of type II unifies (relative) Rota-Baxter operators of any weight, twisted Rota-Baxter operators and Reynolds operators. Then they separately defined the cohomologies, controlling algebras and governing algebras for type I and type II deformation maps. However, in their approach, we observe that
\begin{itemize}
    \item[(i)] neither of these two types of deformation maps unifies embedding tensors on a Lie algebra,
     \item[(ii)] modified Rota-Baxter operators and twisted Rota-Baxter operators both cannot be unified by any particular type (type I or type II) deformation map,
    \item[(iii)] controlling algebras for both type I and type II deformation maps are (curved) $L_\infty$-algebras that can be obtained by similar constructions.
   \end{itemize}
Thus, a natural question that arises from the above discussions is this: 
\begin{center}
    {\em Can we have a single operator that unifies all the operators mentioned above including embedding tensors?}
\end{center}

\medskip

As the underlying algebraic structure of an embedding tensor is given by a Leibniz algebra \cite{kotov,sheng-embed}, one needs to pass to the context of Leibniz algebras. It is also important to remark that all the operators mentioned above (algebra homomorphisms \cite{mandal}, derivations \cite{das-leib-hrs}, crossed homomorphisms \cite{li-wang}, Rota-Baxter operators of any weight \cite{tang-sheng} \cite{tang-sheng-zhou} \cite{das-leib} \cite{das-leib-pub}, modified Rota-Baxter operators \cite{saha}, twisted Rota-Baxter operators \cite{das-guo} and Reynolds operators \cite{das-guo}) are extensively studied on Leibniz algebras in the last few years. They are defined in the same way as in the case of Lie algebras. To unify all these operators on a Leibniz algebra by a single map, one needs to consider a bigger object than a quasi-twilled Leibniz algebra. Namely, we will consider a proto-twilled Leibniz algebra \cite{tang-sheng} and then define deformation maps in a proto-twilled Leibniz algebra. Recall that a proto-twilled Leibniz algebra is a Leibniz algebra $(\mathcal{G}, [~,~]_\mathcal{G})$ whose underlying vector space has a direct sum decomposition $\mathcal{G}= \mathfrak{g} \oplus \mathfrak{h}$ into two subspaces (a quasi-twilled Leibniz algebra is a particular case in which $\mathfrak{h} \subset \mathcal{G}$ is a Leibniz subalgebra). 
A linear map $ r : \mathfrak{h} \rightarrow \mathfrak{g}$ is a {\em deformation map} in the proto-twilled Leibniz algebra if its graph $\mathrm{Gr} (r) \subset \mathfrak{g} \oplus \mathfrak{h}$ is a subalgebra of the Leibniz algebra $( \mathcal{G} = \mathfrak{g} \oplus \mathfrak{h} , [~,~]_\mathcal{G}  )$. A deformation map in a proto-twilled Leibniz algebra unifies all the operators mentioned above on Leibniz algebras and also embedding tensors on Lie algebras. We first show that a deformation map $r: \mathfrak{h} \rightarrow \mathfrak{g}$ induces a Leibniz algebra structure on the vector space $\mathfrak{h}$ (we denote this Leibniz algebra by $\mathfrak{h}_r$). Moreover, there is a suitable representation of this Leibniz algebra $\mathfrak{h}_r$ on the vector space $\mathfrak{g}$. The Loday-Pirashvili cohomology of the Leibniz algebra $\mathfrak{h}_r$ with coefficients in the above representation on $\mathfrak{g}$ is defined to be the cohomology of the deformation map $r$. This cohomology unifies the existing cohomologies of Leibniz algebra homomorphisms \cite{mandal}, derivations \cite{das-leib-hrs}, crossed homomorphisms \cite{li-wang}, (relative) Rota-Baxter operators of any weight \cite{tang-sheng-zhou} \cite{das-leib-pub}, modified Rota-Baxter operators \cite{saha}, twisted Rota-Baxter operators \cite{das-guo} and Reynolds operators \cite{das-guo} on Leibniz algebras and also embedding tensors on Lie algebras \cite{sheng-embed}. Next, given a proto-twilled Leibniz algebra $( \mathcal{G} = \mathfrak{g} \oplus \mathfrak{h} , [~,~]_\mathcal{G}  )$, we construct a curved $L_\infty$-algebra whose Maurer-Cartan elements are precisely deformation maps. This curved $L_\infty$-algebra is called the {\em controlling algebra} for deformation maps. In particular, we get the Maurer-Cartan characterizations of modified Rota-Baxter operators, twisted Rota-Baxter operators and Reynolds operators on Leibniz algebras. Next, twisting the controlling algebra by a Maurer-Cartan element (i.e. a deformation map $r$), one obtains the governing $L_\infty$-algebra that controls the linear deformations of the operator $r$ keeping the underlying proto-twilled Leibniz algebra undeformed. 
It is important to remark that the cohomology (and hence the linear deformation theory) of a deformation map $r$ in a proto-twilled Lie algebra is in general different than the cohomology of the same map $r$ viewed as a deformation map in the proto-twilled Leibniz algebra (cf. Remark \ref{remark-fial}). Finally, we also construct the governing $L_\infty$-algebra that controls the simultaneous deformations of the underlying proto-twilled Leibniz algebra and a given deformation map $r$.

\medskip

The paper is organized as follows. In Section \ref{sec2}, we recall some necessary background on Leibniz algebras and curved $L_\infty$-algebras. In Section \ref{sec3}, we discuss proto-twilled Leibniz algebras and introduce deformation maps in a proto-twilled Leibniz algebra. We observe that a deformation map unifies various well-known operators on a Leibniz algebra and embedding tensors on a Lie algebra. In Section \ref{sec4}, we show that a deformation map induces a new Leibniz algebra structure and there is a suitable representation of this induced Leibniz algebra. Using this, we define the cohomology of a deformation map. In Section \ref{sec5}, we construct a curved $L_\infty$-algebra whose Maurer-Cartan elements are precisely deformation maps in a given proto-twilled Leibniz algebra. We also obtain the governing $L_\infty$-algebra for a fixed deformation map $r$ that controls the linear deformations of the operator $r$. Finally, given a proto-twilled Leibniz algebra and a deformation map $r$, in Section \ref{sec6} we construct the governing $L_\infty$-algebra that controls the simultaneous deformations of the underlying algebra and the operator $r$.

\section{Background on Leibniz algebras and curved $L_\infty$-algebras}\label{sec2}
Here we recall some necessary background on Leibniz algebras and curved $L_\infty$-algebras. Among others, we recall Balavoine's bracket and Voronov's construction of a curved $L_\infty$-algebra. Our main references are \cite{loday-pira,bala,getzler,lada-markl,voro}.

\begin{defn}
A {\bf Leibniz algebra} $(\mathfrak{g}, [~, ~]_\mathfrak{g})$ consists of a vector space $\mathfrak{g}$ equipped with a bilinear operation (called the {\em Leibniz bracket}) $[~,~]_\mathfrak{g} : \mathfrak{g} \times \mathfrak{g} \rightarrow \mathfrak{g}$ satisfying
\begin{align}\label{leib-iden}
    [x, [y, z]_\mathfrak{g}]_\mathfrak{g} = [[x,y]_\mathfrak{g}, z]_\mathfrak{g} + [y, [x, z]_\mathfrak{g}]_\mathfrak{g}, \text{ for all } x, y, z \in \mathfrak{g}.
\end{align}
\end{defn}

A Leibniz algebra defined above is also called a {\em left Leibniz algebra} as the identity (\ref{leib-iden}) is equivalent that all the left translations are derivations for the Leibniz bracket. By keeping this in mind, one may also define the right Leibniz algebras. In this paper, by a Leibniz algebra, we shall always mean a left Leibniz algebra. However, all the results of the present paper can be easily adapted for the right Leibniz algebras without much hard work.

Let $(\mathfrak{g}, [~, ~]_\mathfrak{g})$ be a Leibniz algebra. A {\bf representation} of $(\mathfrak{g}, [~, ~]_\mathfrak{g})$ is given by a triple $(V, \rho^L, \rho^R)$ in which $V$ is a vector space endowed with bilinear maps $\rho^L : \mathfrak{g} \times V \rightarrow V$ and $\rho^R: V \times \mathfrak{g} \rightarrow V$ (called left and right actions respectively) such that for all $x, y \in \mathfrak{g}$ and $v \in V$, the following identities hold:
\begin{align*}
    \rho^L (x, \rho^L (y, v) ) =~& \rho^L ( [x, y]_\mathfrak{g}, v) + \rho^L (y, \rho^L (x, v)),\\
    \rho^L (x, \rho^R (v, y)) =~& \rho^R ( \rho^L (x, v), y) + \rho^R (v, [x, y]_\mathfrak{g}), \\
    \rho^R (v, [x, y]_\mathfrak{g} ) =~& \rho^R (\rho^R (v, x), y) + \rho^L (x, \rho^R (v, y)).
\end{align*}

\begin{exam}\label{exam-ad-coad} Let $(\mathfrak{g}, [~, ~]_\mathfrak{g})$ be any Leibniz algebra.
\begin{itemize}
    \item[(i)] ({\em Adjoint representation}) Then the triple $(\mathfrak{g}, ad^L, ad^R)$ is a representation, where $ad^L, ad^R : \mathfrak{g} \times \mathfrak{g} \rightarrow \mathfrak{g}$ are both given by the Leibniz bracket $[~,~]_\mathfrak{g}$.
    \item[(ii)] ({\em Coadjoint representation}) The triple $(\mathfrak{g}^*, coad^L, coad^R)$ is also a representation, where the maps $coad^L : \mathfrak{g} \times \mathfrak{g}^* \rightarrow \mathfrak{g}^*$ and $coad^R : \mathfrak{g}^* \times \mathfrak{g} \rightarrow \mathfrak{g}^*$ are respectively given by
    \begin{align*}
        coad^L (x, \alpha ) (y) = - \alpha ([x, y]_\mathfrak{g}) ~~~ \text{ and } ~~~ coad^R (\alpha, x) (y) = \alpha ([x, y]_\mathfrak{g} + [y, x]_\mathfrak{g}), \text{ for } x, y \in \mathfrak{g} \text{ and } \alpha \in \mathfrak{g}^*.
    \end{align*}
\end{itemize} 
\end{exam}

We will now recall the cohomology of a Leibniz algebra with coefficients in a representation. Let $(\mathfrak{g}, [~,~]_\mathfrak{g})$ be a Leibniz algebra and $(V, \rho^L, \rho^R)$ be a fixed representation of it. For each $n \geq 0$, define the space of $n$-cochains by $ C^n_\mathrm{Leib} (\mathfrak{g}, V) := \mathrm{Hom} (\mathfrak{g}^{\otimes n}, V)$ and the {coboundary map} $\delta_\mathrm{Leib} : C^n_\mathrm{Leib} (\mathfrak{g}, V) \rightarrow C^{n+1}_\mathrm{Leib} (\mathfrak{g}, V)$ given by
\begin{align*}
   ( \delta_\mathrm{Leib} f) (x_1, \ldots, x_{n+1}) =~& \sum_{i=1}^n (-1)^{i+1} ~ \! \rho^L ( x_i ,  f (x_1, \ldots, \widehat{ x_i} , \ldots, x_{n+1}) ) + (-1)^{n+1} ~ \! \rho^R (f (x_1, \ldots, x_n), x_{n+1} ) \\
   &+  \sum_{1 \leq i < j \leq n+1} (-1)^i ~ \! f (x_1, \ldots, \widehat{x_i}, \ldots, x_{j-1}, [x_i, x_j]_\mathfrak{g}, x_{j+1}, \ldots, x_{n+1}),
\end{align*}
for $f \in  C^n_\mathrm{Leib} (\mathfrak{g}, V)$ and $x_1, \ldots, x_{n+1} \in \mathfrak{g}$. Then  $\{ C^\bullet_\mathrm{Hom} (\mathfrak{g}, V), \delta_\mathrm{Hom} \}$ is a cochain complex and corresponding cohomology groups are called the cohomology groups of the Leibniz algebra $(\mathfrak{g}, [~,~]_\mathfrak{g})$ with coefficients in the representation $(V, \rho^L, \rho^R)$.

\medskip

\noindent {\em Balavoine bracket.} Let $\mathfrak{g}$ be any vector space. Given any $f \in \mathrm{Hom}(\mathfrak{g}^{\otimes m}, \mathfrak{g})$ and $g \in \mathrm{Hom}(\mathfrak{g}^{\otimes n}, \mathfrak{g})$, their {\bf Balavoine bracket} $\llbracket f, g \rrbracket_\mathsf{B} \in \mathrm{Hom}(\mathfrak{g}^{\otimes m+n-1}, \mathfrak{g})$ is defined by
\begin{align*}
    \llbracket f, g \rrbracket_\mathsf{B} := f \diamond g - (-1)^{(m-1) (n-1)} ~ \! g \diamond f,
\end{align*}
where
\begin{align*}
    (f \diamond g) (x_1, \ldots, x_{m+n-1} ) = \sum_{i=1}^m ~& (-1)^{(i-1) (n-1)} \sum_{\sigma \in \mathbb{S}_{(i-1, n-1)} } (-1)^\sigma \\
    & f \big(  x_{\sigma (1)}, \ldots, x_{\sigma (i-1)}, g \big( x_{\sigma (i)}, \ldots, x_{\sigma (i+n-2)}, x_{i+n-1}   \big), x_{i+n}, \ldots, x_{m+n-1}   \big),
\end{align*}
for $x_1, \ldots, x_{m+n-1} \in \mathfrak{g}$. Then the following result is well-known \cite{bala}.

\begin{thm}
\begin{itemize}
    \item[(i)]  Let $\mathfrak{g}$ be a vector space. Then the graded vector space $ \oplus_{n=0}^\infty \mathrm{Hom} (\mathfrak{g}^{\otimes n +1}, \mathfrak{g})$ endowed with the Balavoine bracket $\llbracket ~,~ \rrbracket_\mathsf{B}$ is a graded Lie algebra, called the Balavoinve's graded Lie algebra associated to the vector space $\mathfrak{g}$.
    \item[(ii)]  There is a one-to-one correspondence between Leibniz algebra structures on the vector space $\mathfrak{g}$ and Maurer-Cartan elements of the Balavoine's graded Lie algebra $\big(   \oplus_{n=0}^\infty \mathrm{Hom} (\mathfrak{g}^{\otimes n +1}, \mathfrak{g}) , \llbracket ~, ~ \rrbracket_\mathsf{B} \big)$.
\end{itemize}
\end{thm}

\medskip

A {\bf curved $L_\infty$-algebra} \cite{getzler,lada-markl} is a pair $(\mathcal{L}, \{ l_k \}_{k=0}^\infty)$ consisting of a graded vector space $\mathcal{L} = \oplus_{n \in \mathbb{Z}} \mathcal{L}_n$ equipped with a collection $\{ l_k : \mathcal{L}^{\otimes k} \rightarrow \mathcal{L} \}_{k=0}^\infty$ of  degree $1$ graded symmetric linear maps satisfying the following set of identities:
\begin{align*}
    \text{for each } N \geq 0, \qquad \displaystyle\sum^{N}_{i=0}~\sum_{\sigma\in \mathbb{S}_{(i,N-i)}} \epsilon(\sigma)~ \! l_{N-i+1}\big(l_i(x_{\sigma(1)},\ldots,x_{\sigma(i)}),x_{\sigma(i+1)},\ldots,x_{\sigma(N)}\big) = 0, 
\end{align*}
for homogeneous elements $x_1, \ldots, x_N \in \mathcal{L}$. Here $ \epsilon (\sigma) = \epsilon (\sigma; x_1, \ldots, x_N)$ is the Koszul sign that appears in the graded context.

In the above definition, $l_0$ is simply a degree $1$ homogeneous element of $\mathcal{L}$ (i.e. $l_0 \in \mathcal{L}_1$). When $l_0 = 0$, a curved $L_\infty$-algebra becomes an $L_\infty$-algebra. All (curved) $L_\infty$-algebras considered in this paper are assumed to have a complete descending filtration \cite{getzler,jonas} so that all infinite sums under consideration are convergent.

Let $(\mathcal{L}, \{ l_k \}_{k=0}^\infty)$ be a (curved) $L_\infty$-algebra. An element $\alpha \in \mathcal{L}_0$ is said to be a {\bf Maurer-Cartan element} of the (curved) $L_\infty$-algebra if $\alpha$ satisfies
\begin{align*}
    l_0 + \displaystyle\sum_{k=1}^{\infty} \dfrac{1}{k!} ~ \! l_k(\alpha,\ldots,\alpha) = 0.
\end{align*}
A Maurer-Cartan element of a (curved) $L_\infty$-algebra can be used to construct a new $L_\infty$-algebra by twisting. More precisely, we have the following result \cite{getzler}.

\begin{thm}\label{mc-twist-thm}
    Let $(\mathcal{L}, \{ l_k \}_{k=0}^\infty)$ be a (curved) $L_\infty$-algebra and $\alpha$ be a Maurer-Cartan element of it. Then $(\mathcal{L}, \{ l_k^\alpha \}_{k=1}^\infty)$ is an $L_\infty$-algebra, where
    \begin{align*}
         {l^\alpha_k}(x_1,\ldots,x_k) :=  \displaystyle\sum_{n=0}^{\infty} \dfrac{1}{n!} ~ \! l_{n+k}(\alpha,\ldots,\alpha,x_1,\ldots,x_k), \text{ for } k \geq 1 \text{ and } x_1,\ldots, x_k\in \mathcal{L}.
    \end{align*}
\end{thm}

A useful construction of a (curved) $L_\infty$-algebra is given by Voronov \cite{voro}. First recall that a {\bf curved $V$-data} is a quadruple $(\mathfrak{B}, \mathfrak{a}, P, \Delta)$ consisting of a graded Lie algebra $(\mathfrak{B}, \llbracket ~, ~ \rrbracket)$, an abelian graded Lie subalgebra $\mathfrak{a} \subset \mathfrak{B}$, a projection map $P : \mathfrak{B} \rightarrow \mathfrak{a} \subset \mathfrak{B}$ whose kernel $\mathrm{ker}(P) \subset \mathfrak{B}$ is a graded Lie subalgebra and an element $\Delta \in \mathfrak{B}_1$ that satisfies $\llbracket \Delta, \Delta \rrbracket = 0$. A curved $V$-data  $(\mathfrak{B}, \mathfrak{a}, P, \Delta)$ is called a {\em $V$-data} if $\Delta \in \mathrm{ker}(P)_1$.

\begin{thm}\label{thm-voro}
    Let  $(\mathfrak{B}, \mathfrak{a}, P, \Delta)$ be a curved $V$-data. Then $(\mathfrak{a}, \{ l_k \}_{k=0}^\infty)$ is a curved $L_\infty$-algebra, where
    \begin{align*}
        l_0= P(\Delta) \quad \text{ and } \quad 
  l_k(a_1,\ldots,a_k) := P \llbracket \cdots \llbracket \llbracket \Delta,a_1 \rrbracket ,a_2 \rrbracket ,\ldots,a_k \rrbracket, \text{ for } k \geq 1 \text{ and } a_1,\ldots,a_k \in \mathfrak{a}.    
    \end{align*}
    More generally, if $\mathfrak{B}' \subset \mathfrak{B}$ is a graded Lie subalgebra  with $\llbracket \Delta, \mathfrak{B}' \rrbracket \subset \mathfrak{B}'$ then the pair $( s^{-1} \mathfrak{B}' \oplus \mathfrak{a} , \{ \widetilde{l}_k \}_{k=0}^\infty)$ is a curved $L_\infty$-algebra, where the structure maps are given by (up to permutations of the inputs)
    \begin{align*}
         \widetilde{l}_0=~& (0, P(\Delta)),\\
\widetilde{l}_1( (s^{-1} x, 0) ) =~& (s^{-1} \llbracket \Delta,x \rrbracket , P(x)),\\ 
\widetilde{l}_2( (s^{-1} x, 0), (s^{-1}y, 0) ) =~& ((-1)^{\lvert x \rvert} ~ \! s^{-1} \llbracket x,y \rrbracket,0 ),\\
\widetilde{l}_k\big( (s^{-1} x,0), (0,a_1), \ldots, (0, a_{k-1} )\big) =~& (0, P~ \! \llbracket \cdots \llbracket \llbracket x,a_1 \rrbracket,a_2\rrbracket,\ldots,a_{k-1} \rrbracket),~~k\geq 2,\\
\widetilde{l}_k\big( (0,a_1),\ldots, (0,a_{k})\big) =~& (0, P~ \! \llbracket \cdots \llbracket \llbracket \Delta,a_1 \rrbracket,a_2\rrbracket,\ldots,a_{k} \rrbracket),~~ k\geq 1,
    \end{align*}
    for homogeneous elements $x, y \in \mathfrak{B}'$ (considered as elements $s^{-1} x, s^{-1} y \in s^{-1} \mathfrak{B}'$ by a shifting) and $a_1, \ldots, a_k \in \mathfrak{a}$.
\end{thm}

\section{Proto-twilled Leibniz algebras and deformation maps}\label{sec3}
In this section, we first consider proto-twilled Leibniz algebras \cite{tang-sheng} and introduce deformation maps in a proto-twilled Leibniz algebra. We observe that deformation maps in a proto-twilled Leibniz algebra unify Leibniz algebra homomorphisms, derivations, crossed homomorphisms, (relative) Rota-Baxter operators of any weight, modified Rota-Baxter operators, twisted Rota-Baxter operators, Reynolds operators on Leibniz algebras and also embedding tensors on Lie algebras.

\begin{defn}
    A {\bf proto-twilled Leibniz algebra} is a Leibniz algebra $(\mathcal{G}, [~,~]_\mathcal{G})$ whose underlying vector space has a direct sum decomposition $\mathcal{G} = \mathfrak{g} \oplus \mathfrak{h}$ into two subspaces.
\end{defn}

It is said to be a {\em quasi-twilled Leibniz algebra} if $\mathfrak{h} \subset \mathcal{G}$ is a Leibniz subalgebra and a {\em twilled Leibniz algebra} (also called a {\em matched pair of Leibniz algebras}) if $\mathfrak{g} \subset \mathcal{G}$ and $\mathfrak{h} \subset \mathcal{G}$ are both Leibniz subalgebras. A (proto-, quasi-)twilled Leibniz algebra as above is often denoted by $( \mathcal{G} = \mathfrak{g} \oplus \mathfrak{h}, [~,~]_\mathcal{G})$ if the notations are understood.

\begin{exam}\label{dir-product}
    Let $(\mathfrak{g}, [~,~]_\mathfrak{g})$ and $(\mathfrak{h}, [~,~]_\mathfrak{h})$ be two Leibniz algebras. Then the direct product $(\mathfrak{g} \oplus \mathfrak{h}, [~,~]_\mathrm{dir})$ is a proto-twilled Leibniz algebra, where
    \begin{align*}
        [ (x, u), (y, v)]_\mathrm{dir} = ( [x, y]_\mathfrak{g}, [u, v]_\mathfrak{h}), \text{ for } (x,u), (y, v) \in \mathfrak{g} \oplus \mathfrak{h}.
    \end{align*}
\end{exam}

\begin{exam}\label{semi-proto}
    Let $(\mathfrak{g}, [~,~]_\mathfrak{g})$ be a Leibniz algebra and $(V, \rho^L, \rho^R)$ be a representation of it. Then the semidirect product $(\mathfrak{g} \oplus V, [~, ~]_\ltimes)$ is a proto-twilled Leibniz algebra, where
    \begin{align*}
        [ (x, u), (y, v)]_\ltimes = ( [x, y]_\mathfrak{g} ~ \! , ~ \! \rho^L (x, v) + \rho^R (u, y)), \text{ for } (x, u), (y, v) \in \mathfrak{g} \oplus V.
    \end{align*}
    Symmetrically, $(V \oplus \mathfrak{g}, [~,~]_\rtimes)$ is also a proto-twilled Leibniz algebra, where 
    \begin{align}\label{deri-proto}
        [(u, x) ,(v, y) ]_\rtimes = ( \rho^L (x, v) + \rho^R (u, y) ~ \!, ~ \! [x, y]_\mathfrak{g} ), \text{ for } (u, x), (v, y) \in V \oplus \mathfrak{g}.
    \end{align}
\end{exam}

\begin{exam}\label{rota-1-proto}
   Let $(\mathfrak{g}, [~,~]_\mathfrak{g})$ and $(\mathfrak{h}, [~,~]_\mathfrak{h})$ be two Leibniz algebras. Suppose there are bilinear maps $\rho^L : \mathfrak{g} \times \mathfrak{h} \rightarrow \mathfrak{h}$ and $\rho^R : \mathfrak{h} \times \mathfrak{g} \rightarrow \mathfrak{h}$ that makes the triple $(\mathfrak{h}, \rho^L, \rho^R)$ into a representation of the Leibniz algebra $(\mathfrak{g}, [~,~]_\mathfrak{g})$ satisfying additionally
   \begin{align}
       \rho^L (x, [u, v]_\mathfrak{h}) =~& [ \rho^L (x, u), v]_\mathfrak{h} + [u, \rho^L (x, v) ]_\mathfrak{h}, \label{leib-act-1}\\
       [u, \rho^L (x, v )]_\mathfrak{h}  =~& [ \rho^R (u, x), v ]_\mathfrak{h} + \rho^R (x, [u, v]_\mathfrak{h}), \\
       [ u, \rho^R (v, x)]_\mathfrak{h} =~& \rho^R ( [u, v]_\mathfrak{h}, x)+ [v, \rho^R (u, x)]_\mathfrak{h}, \label{leib-act-3}
   \end{align}
   for all $x \in \mathfrak{g}$ and $u, v \in \mathfrak{h}$. Then  the semidirect product $(\mathfrak{g} \oplus \mathfrak{h}, [~,~]_\ltimes)$ is a proto-twilled Leibniz algebra, where
   \begin{align*}
       [(x, u), (y, v)]_\ltimes = (  [x, y]_\mathfrak{g} ~ \!  , ~ \! [u, v]_\mathfrak{h} + \rho^L (x, v) + \rho^R (u, y) ), \text{ for } (x, u), (y, v) \in \mathfrak{g} \oplus \mathfrak{h}.
   \end{align*}
   Symmetrically, $(\mathfrak{h} \oplus \mathfrak{g}, [~,~]_\rtimes)$ is a proto-twilled Leibniz algebra, where
   \begin{align}\label{crossed-proto}
       [(u, x), (v, y)]_\rtimes = ( [u, v]_\mathfrak{h} + \rho^L (x, v) + \rho^R (u, y)  ~\! , ~ \! [x, y]_\mathfrak{g}), \text{ for } (u, x), (v, y) \in \mathfrak{h} \oplus \mathfrak{g}.
   \end{align}
\end{exam}

\begin{exam}\label{mod-proto}
    Let $(\mathfrak{g}, [~,~]_\mathfrak{g})$ be a Leibniz algebra. Then $(\mathfrak{g} \oplus \mathfrak{g}, [~,~]_\mathrm{mod})$ is a proto-twilled Leibniz algebra, where
    \begin{align*}
        [ (x', x), (y', y)]_\mathrm{mod} = ( [x', y']_\mathfrak{g} + [x, y]_\mathfrak{g} ~ \!, ~ \!  [x', y]_\mathfrak{g} + [x, y']_\mathfrak{g} ), \text{ for } (x', x), (y', y) \in \mathfrak{g} \oplus \mathfrak{g}.
    \end{align*}
\end{exam}

\begin{exam}\label{twisted-proto}
    Let $(\mathfrak{g}, [~,~]_\mathfrak{g})$ be a Leibniz algebra and $(V, \rho^L, \rho^R)$ be a representation of it. Suppose $\theta \in Z^2_\mathrm{Leib} (\mathfrak{g}, V)$ is a Leibniz $2$-cocycle (i.e. $\theta \in C^2_\mathrm{Leib} (\mathfrak{g}, V)$ with $\delta_\mathrm{Leib} \theta = 0$). Then $(\mathfrak{g} \oplus V, [~,~]_{\ltimes_\theta})$ is a proto-twilled Leibniz algebra, where
    \begin{align*}
        [ (x, u), (y, v) ]_{\ltimes_\theta} = \big( [x, y]_\mathfrak{g} ~ \! , ~ \! \rho^L (x, v) + \rho^R (u, y) + \theta (x, y) \big), \text{ for } (x, u), (y, v) \in \mathfrak{g} \oplus V.
    \end{align*}
\end{exam}

\begin{exam}\label{rey-proto}
    Let $(\mathfrak{g}, [~,~]_\mathfrak{g})$ be a Leibniz algebra. Then $(\mathfrak{g} \oplus \mathfrak{g}, [~,~]_{\ltimes_{-}})$ is a proto-twilled Leibniz algebra, where
    \begin{align*}
        [ (x', x), (y', y)]_{\ltimes_{-}} = ( [x', y']_\mathfrak{g} ~ \! , ~ \! [x', y]_\mathfrak{g} + [x, y']_\mathfrak{g} - [x', y']_\mathfrak{g}), \text{ for } (x', x) , (y', y) \in \mathfrak{g} \oplus \mathfrak{g}.
    \end{align*}
\end{exam}

\begin{exam}\label{hemi-proto}
    Let $(\mathfrak{g}, [~,~]_\mathfrak{g})$ be a Lie algebra and $(V, \rho)$ be a representation of it. Then the hemi-semidirect product $(\mathfrak{g} \oplus V, [~,~]_\mathrm{hemi})$ is a proto-twilled Leibniz algebra, where
    \begin{align*}
        [(x, u), (y, v) ]_\mathrm{hemi} = ( [x, y]_\mathfrak{g} ~ \!, ~ \! \rho (x)v), \text{ for } (x, u), (y, v) \in \mathfrak{g} \oplus V. 
    \end{align*}
\end{exam}

More generally, any twilled Leibniz algebras and quasi-twilled Leibniz algebras are examples of proto-twilled Leibniz algebras.

\medskip

Let $\mathfrak{g}$ and $\mathfrak{h}$ be two vector spaces. Let $\mathcal{G}^{k, l}$ be the direct sum of all possible $k+l$ tensor powers of $\mathfrak{g}$ and $\mathfrak{h}$ in which $\mathfrak{g}$ appears $k$ times and $\mathfrak{h}$ appears $l$ times. For example,
\begin{align*}
    \mathcal{G}^{1,1} = (\mathfrak{g} \otimes \mathfrak{h}) \oplus (\mathfrak{h} \otimes \mathfrak{g}) ~~~~ \text{ and } ~~~~ \mathcal{G}^{2,1} = (\mathfrak{g} \otimes \mathfrak{g} \otimes \mathfrak{h}) \oplus (\mathfrak{g} \otimes \mathfrak{h} \otimes \mathfrak{g}) \oplus (\mathfrak{h} \otimes \mathfrak{g} \otimes \mathfrak{g}) ~~ \text{etc.}
\end{align*}
Thus, if we write $\mathcal{G} = \mathfrak{g} \oplus \mathfrak{h}$ then for any $n \geq 0$, we have $\mathcal{G}^{\otimes n+1} = \oplus_{\substack{k + l = n+1\\k, l \geq 0}} \mathcal{G}^{k, l}$. For any linear map $c : \mathcal{G}^{k,l} \rightarrow \mathfrak{g}$ (or $c : \mathcal{G}^{k,l} \rightarrow \mathfrak{h}$), one can define a linear map $\widetilde{c} \in \mathrm{Hom} (\mathcal{G}^{\otimes k + l}, \mathcal{G})$ by
\begin{align*}
    \widetilde{c} = \begin{cases}
c & \text{ on } \mathcal{G}^{k,l} \\
0 & \text{ otherwise}.
    \end{cases}
\end{align*}
The map $\widetilde{c}$ is called the {\em horizontal lift} of $c$. Recall from \cite{tang-sheng} that a map $f \in \mathrm{Hom} (\mathcal{G}^{\otimes n+1}, \mathcal{G})$ is said to have {\em bidegree} $k|l$ with $k, l \geq -1$ and $k+l = n$ if the map $f$ satisfies
\begin{align*}
    f (\mathcal{G}^{k+1, l}) \subset \mathfrak{g}, \quad f ( \mathcal{G}^{k, l+1}) \subset \mathfrak{h} ~~~~ \text{ and } ~~~~ f = 0 \text{ otherwise}.
\end{align*}
We denote the set of all maps of bidegree $k|l$ by $C^{k|l}$. Then we have the following isomorphism $C^{k|l} \cong \mathrm{Hom} (\mathcal{G}^{k+1, l}, \mathfrak{g}) \oplus \mathrm{Hom} (\mathcal{G}^{k, l+1}, \mathfrak{h})$. In one direction, this isomorphism is given by the restriction and in the other direction, it is given by the horizontal lifting map. Hence
\begin{align*}
    \mathrm{Hom} (\mathcal{G}^{\otimes n+1}, \mathcal{G}) = \oplus_{\substack{ k+l=n+1 \\ k , l \geq 0}} \big(  \mathrm{Hom} (\mathcal{G}^{k, l}, \mathfrak{g}) \oplus \mathrm{Hom} (\mathcal{G}^{k, l}, \mathfrak{h}) \big) \cong C^{n+1 | -1} \oplus C^{n|0} \oplus \cdots \oplus C^{0 | n} \oplus C^{-1 | n+1}.
\end{align*}

The following lemma has been proved in \cite{tang-sheng}.

\begin{lemma}\label{lemma-tang}
    Let $f \in C^{k_f | l_f}$ and $g \in C^{ k_g | l_g}$. Then their Balavoine bracket $\llbracket f, g \rrbracket_\mathsf{B} \in C^{k_f + l_f | k_g + l_g}$.
\end{lemma}

As a consequence of the above lemma, we get the following.

\begin{prop}\label{prop-mqr}
    Let $\mathfrak{g}$ and $\mathfrak{h}$ be two vector spaces. For each $n \geq 0$, we define
    \begin{align*}
        \mathcal{M}_n :=~& C^{n | 0} \oplus \cdots \oplus C^{0 | n} \subset \mathrm{Hom} (\mathcal{G}^{\otimes n+1}, \mathcal{G}),\\
        \mathcal{Q}_n :=~& C^{n+1|-1} \oplus  C^{n | 0} \oplus \cdots \oplus C^{0 | n} = C^{n+1 | -1} \oplus \mathcal{M}_n \subset \mathrm{Hom} (\mathcal{G}^{\otimes n+1}, \mathcal{G}),\\
        \mathcal{R}_n :=~& C^{n | 0} \oplus \cdots \oplus C^{0 | n} \oplus C^{-1 | n+1} = \mathcal{M}_n \oplus C^{-1 | n+1} \subset \mathrm{Hom} (\mathcal{G}^{\otimes n+1}, \mathcal{G}).
    \end{align*}
    Then $\oplus_{n=0}^\infty \mathcal{M}_n$, $\oplus_{n=0}^\infty \mathcal{Q}_n$ and $\oplus_{n=0}^\infty \mathcal{R}_n$ are all graded Lie subalgebras of $\big( \oplus_{n=0}^\infty \mathrm{Hom} (\mathcal{G}^{\otimes n+1}, \mathcal{G}) , \llbracket ~,~ \rrbracket_\mathsf{B} \big).$
\end{prop}

\medskip

Next, suppose there is a map $\Omega = [~,~]_\mathcal{G} \in \mathrm{Hom} ( \mathcal{G}^{\otimes 2}, \mathcal{G})$.
Note that $[~,~]_\mathcal{G}$ can be written as
\begin{align*}
    [ (x , u) , (y, v) ]_\mathcal{G} = \big(  [x, y]_\mathfrak{g} + \psi^R (x, v) + \psi^L (u, y) + \eta (u, v) ~ \! , ~ \! [u, v]_\mathfrak{h}  + \rho^L (x, v) + \rho^R (u, y) + \theta (x, y) \big),
\end{align*}
for $(x, u), (y, v) \in \mathfrak{g} \oplus \mathfrak{h}$, where
\begin{align*}
    &[~,~]_\mathfrak{g} : \mathfrak{g} \times \mathfrak{g} \rightarrow \mathfrak{g}, \quad [~,~]_\mathfrak{h} :  \mathfrak{h} \times \mathfrak{h} \rightarrow \mathfrak{h}, \quad \rho^L :  \mathfrak{g} \times \mathfrak{h} \rightarrow \mathfrak{h}, \quad \rho^R :  \mathfrak{h} \times \mathfrak{g} \rightarrow \mathfrak{h},\\
    &\psi^L :  \mathfrak{h} \times \mathfrak{g} \rightarrow \mathfrak{g}, \qquad \psi^R :  \mathfrak{g} \times \mathfrak{h} \rightarrow \mathfrak{g}, \qquad \theta :  \mathfrak{g} \times \mathfrak{g} \rightarrow \mathfrak{h}, \quad \eta :  \mathfrak{h} \times \mathfrak{h} \rightarrow \mathfrak{g}
\end{align*}
are bilinear maps that are uniquely characterized by
\begin{align*}
    [ (x,0), (y, 0)]_\mathcal{G} =~& ( [x, y]_\mathfrak{g}, \theta (x, y)),  \qquad &&[(0, u), (0, v)]_\mathcal{G} = ( \eta (u, v), [u, v]_\mathfrak{h}), \\
    [(x, 0), (0, u)]_\mathcal{G} =~& ( \psi^R (x, u), \rho^L (x, u)), \quad && [(0, u), (x, 0)]_\mathcal{G} = ( \psi^L (u, x) , \rho^R (u, x)),
\end{align*}
for $x, y \in \mathfrak{g}$ and $u, v \in \mathfrak{h}$. On the other hand, since $\Omega = [~,~]_\mathcal{G} \in \mathrm{Hom} (\mathcal{G}^{\otimes 2}, \mathcal{G}) = C^{2 | -1} \oplus C^{1| 0} \oplus C^{0|1} \oplus C^{-1 | 2}$, it has components in $C^{2 | -1}$, $C^{1|0}$, $C^{0|1}$ and $C^{-1|2}$. Explicitly, we have
\begin{align*}
   \Omega = [~,~]_\mathcal{G} = \widetilde{\theta} + \mu + \nu  + \widetilde{\eta},
\end{align*}
where $\mu \in C^{1|0}$ and $\nu \in C^{0|1}$ are respectively given by
\begin{align}
    \mu (( x, u), (y,v)) :=~& (\widetilde{[~,~]_\mathfrak{g}} + \widetilde{\rho^L} + \widetilde{\rho^R})  (( x, u), (y,v)) = ( [x, y]_\mathfrak{g} ~ \! , ~ \! \rho^L (x, v) + \rho^R (u, y)), \label{mu-nu1} \\ \nu ((x, u), (y, v)) :=~&  (\widetilde{[~,~]_\mathfrak{h}} + \widetilde{\psi^L} + \widetilde{\psi^R}) (( x, u), (y,v)) =  ( \psi^R (x, v) + \psi^L (u, y) ~ \! , ~ \! [u, v]_\mathfrak{h}). \label{mu-nu2}
\end{align}

With the above notations, $(\mathcal{G} = \mathfrak{g} \oplus \mathfrak{h}, [~,~]_\mathcal{G})$ is a proto-twilled Leibniz algebra (i.e. $[~,~]_\mathcal{G}$ satisfies the Leibniz identity or equivalently, $\llbracket \Omega , \Omega \rrbracket_\mathsf{B} = 0$) if and only if 
\begin{align*}
    \llbracket \mu, \widetilde{\theta} \rrbracket_\mathsf{B} =~& 0, \\
    \llbracket \mu, \mu \rrbracket_\mathsf{B} + 2  \llbracket \nu, \widetilde{\theta} \rrbracket_\mathsf{B} =~& 0, \\
     \llbracket \mu, \nu \rrbracket_\mathsf{B} + \llbracket \widetilde{\theta}, \widetilde{\eta} \rrbracket_\mathsf{B} =~& 0,\\
      \llbracket \nu, \nu \rrbracket_\mathsf{B} + 2  \llbracket \mu, \widetilde{\eta} \rrbracket_\mathsf{B} =~& 0,\\
       \llbracket \nu, \widetilde{\eta} \rrbracket_\mathsf{B} =~& 0.
\end{align*}





\begin{defn}\label{defn-defor-map}
    A linear map $r : \mathfrak{h} \rightarrow \mathfrak{g}$ is said to be a {\bf deformation map} in a proto-twilled Leibniz algebra $(\mathcal{G} = \mathfrak{g} \oplus \mathfrak{h}, [~,~]_\mathcal{G})$ if it satisfies
    \begin{align}\label{defor-map-iden}
    [ r (u), r (v)]_\mathfrak{g} + \psi^R (r (u), v) + \psi^L ( u, r (v)) + \eta (u, v) = r \big( [u, v]_\mathfrak{h} + \rho^L ( r (u), v) + \rho^R (u, r (v)) + \theta ( r (u), r (v))   \big),
\end{align}
for all $u, v \in \mathfrak{h}$.
\end{defn}

In the following, we will show that a wide class of operators on Leibniz algebras can be realized as deformation maps in suitable proto-twilled Leibniz algebras.

\begin{exam}
    ({Leibniz algebra homomorphisms}) Let $(\mathfrak{g}, [~,~]_\mathfrak{g})$ and $(\mathfrak{h}, [~,~]_\mathfrak{h})$ be two Leibniz algebras. A linear map $r : \mathfrak{h} \rightarrow \mathfrak{g}$ is a deformation map in the direct product $(\mathfrak{g} \oplus \mathfrak{h}, [~,~]_\mathrm{dir})$ given in Example \ref{dir-product} if and only if $r$ is a {\em Leibniz algebra homomorphism}, i.e.
    \begin{align*}
        r ([u, v]_\mathfrak{h}) = [ r (u), r (v)]_\mathfrak{g}, \text{ for all } u , v \in \mathfrak{h}.
    \end{align*}
\end{exam}

\begin{exam}\label{exam-relative}
    (Relative Rota-Baxter operators of weight $0$) Let $(\mathfrak{g}, [~,~]_\mathfrak{g})$ be a Leibniz algebra and $(V, \rho^L, \rho^R)$ be a representation of it. Then a linear map $r : V \rightarrow \mathfrak{g}$ is a deformation map in the semidirect product Leibniz algebra $(\mathfrak{g} \oplus V, [~,~]_\ltimes)$ given in Example \ref{semi-proto} if and only if $r$ is a {\em relative Rota-Baxter operator of weight $0$} \cite{tang-sheng,tang-sheng-zhou}, i.e.
    \begin{align*}
        [ r(u),  r (v) ]_\mathfrak{g} = r ( \rho^L (r (u), v) + \rho^R (u, r(v))), \text{ for all } u, v \in V.
    \end{align*}

    Let $(\mathfrak{g}, [~,~]_\mathfrak{g})$ be a Leibniz algebra. An element $s \in \mathrm{Sym}^2 (\mathfrak{g})$ is said to be a {\bf classical Leibniz r-matrix} if it satisfies the {\em classical Leibniz Yang-Baxter equation} $[ [ s, s ]  ] = 0$ where the explicit description of the bracket $[[ ~, ~ ]]$ is given in \cite{tang-sheng}. It has been shown that an element $s \in \mathrm{Sym}^2 (\mathfrak{g})$ is a classical Leibniz r-matrix if and only if the map
\begin{align*}
    s^\sharp : \mathfrak{g}^* \rightarrow \mathfrak{g} ~~ \text{ defined by } ~~ s^\sharp (\alpha) (\beta) := s (\alpha, \beta), \text{ for } \alpha , \beta \in \mathfrak{g}^*
\end{align*}
is a relative Rota-Baxter operator of weight $0$ (where $\mathfrak{g}^*$ is endowed with the coadjoint representation (see Example \ref{exam-ad-coad})). Therefore, a classical Leibniz r-matrix gives rise to a deformation map in the semidirect product Leibniz algebra $(\mathfrak{g} \oplus \mathfrak{g}^*, [~,~ ]_\ltimes)$.

\end{exam}

\begin{exam}
    ({Derivations}) Let $(\mathfrak{g}, [~,~]_\mathfrak{g})$ be a Leibniz algebra and $(V, \rho^L, \rho^R)$ be a representation of it. Consider the proto-twilled Leibniz algebra $(V \oplus \mathfrak{g}, [~,~]_\rtimes)$ given in (\ref{deri-proto}). Then a linear map $ r: \mathfrak{g} \rightarrow V$ is a deformation map if and only if $r$ is a {\em derivation} \cite{das-leib-hrs}, i.e.
    \begin{align*}
        r ([x, y]_\mathfrak{g} ) = \rho^L (x, r (y)) + \rho^R ( r (x), y), \text{ for all } x, y \in \mathfrak{g}.
    \end{align*}
\end{exam}

\begin{exam}
    (Relative Rota-Baxter operators of weight $1$)  Let $(\mathfrak{g}, [~,~]_\mathfrak{g})$ and $(\mathfrak{h}, [~,~]_\mathfrak{h})$ be two Leibniz algebras. Suppose there are bilinear maps $\rho^L : \mathfrak{g} \times \mathfrak{h} \rightarrow \mathfrak{h}$ and $\rho^R : \mathfrak{h} \times \mathfrak{g} \rightarrow \mathfrak{h}$ that makes the triple $(\mathfrak{h}, \rho^L, \rho^R)$ into a representation of the Leibniz algebra $(\mathfrak{g}, [~,~]_\mathfrak{g})$ satisfying additionally (\ref{leib-act-1})-(\ref{leib-act-3}). Then a linear map $ r : \mathfrak{h} \rightarrow \mathfrak{g}$ is a deformation map in the semidirect product $(\mathfrak{g} \oplus \mathfrak{h}, [~,~]_\ltimes)$ given in Example \ref{rota-1-proto} if and only if $r$ is a {\em relative Rota-Baxter operator of weight $1$} \cite{das-leib-pub}, i.e.
    \begin{align*}
        [ r(u),  r (v) ]_\mathfrak{g} = r ( [u, v]_\mathfrak{h} + \rho^L (r (u), v) + \rho^R (u, r(v))), \text{ for all } u, v \in \mathfrak{h}.
    \end{align*}
\end{exam}

\begin{exam}
    (Crossed homomorphisms) With the assumptions given in the previous example, we consider the proto-twilled Leibniz algebra $(\mathfrak{h} \oplus \mathfrak{g}, [~,~]_\rtimes)$ given in (\ref{crossed-proto}).
    Then a linear map $r : \mathfrak{g} \rightarrow \mathfrak{h}$ is a deformation map in $(\mathfrak{h} \oplus \mathfrak{g}, [~,~]_\rtimes)$ if and only if $r$ is a {\em crossed homomorphism} \cite{li-wang}, i.e.
    \begin{align*}
        r ( [x, y]_\mathfrak{g}) = [ r(x), r (y)]_\mathfrak{h} + \rho^L (x, r(y)) + \rho^R ( r(x), y), \text{ for } x, y \in \mathfrak{g}.
    \end{align*}
\end{exam}

\begin{exam}
    (Modified Rota-Baxter operators) Let $(\mathfrak{g}, [~,~]_\mathfrak{g})$ be a Leibniz algebra. Then a linear map $r : \mathfrak{g} \rightarrow \mathfrak{g}$ is a deformation map in the proto-twilled Leibniz algebra $(\mathfrak{g} \oplus \mathfrak{g}, [~,~]_\mathrm{mod})$ given in Example \ref{mod-proto} if and only if $r$ satisfies
    \begin{align*}
        [r(x), r(y)]_\mathfrak{g} = r ( [r(x), y]_\mathfrak{g} + [x, r(y)]_\mathfrak{g}) - [x, y]_\mathfrak{g}, \text{ for all } x, y \in \mathfrak{g}.
    \end{align*}
    This is equivalent to $r$ being a {\em modified Rota-Baxter operator} \cite{saha}.
\end{exam}

\begin{exam}
    (Twisted Rota-Baxter operators) Let $( \mathfrak{g}, [~, ~]_\mathfrak{g}) $ be a Leibniz algebra and $(V, \rho^L, \rho^R)$ be a representation of it. For any Leibniz $2$-cocycle $\theta \in Z^2_\mathrm{Leib} (\mathfrak{g}, V)$, we consider the proto-twilled Leibniz algebra $(\mathfrak{g} \oplus V, [~,~]_{\ltimes_\theta})$ given in Example \ref{twisted-proto}. Then a linear map $r : V \rightarrow \mathfrak{g}$ is a deformation map if and only if it satisfies
    \begin{align*}
        [ r (u), r(v)]_\mathfrak{g} = r \big(  \rho^L ( r(u), v) + \rho^R (u, r(v)) + \theta ( r(u), r(v)) \big), \text{ for all } u, v \in V.
    \end{align*}
    In other words, $r$ is a {\em $\theta$-twisted Rota-Baxter operator} \cite{das-guo}.
\end{exam}

\begin{exam}
    (Reynolds operators) Let $(\mathfrak{g}, [~,~]_\mathfrak{g})$ be a Leibniz algebra. Consider the proto-twilled Leibniz algebra $(\mathfrak{g} \oplus \mathfrak{g}, [~,~]_{\ltimes_{-}})$ given in Example \ref{rey-proto}. For this proto-twilled Leibniz algebra, a linear map $r : \mathfrak{g} \rightarrow \mathfrak{g}$ is a deformation map if and only if $r$ is a {\em Reynolds operator} \cite{das-guo}, i.e.
    \begin{align*}
        [ r(x) , r (y)]_\mathfrak{g} = r  ( [r(x), y]_\mathfrak{g} + [x, r(y)]_\mathfrak{g} - [ r (x), r (y)]_\mathfrak{g} ), \text{ for all } x, y \in \mathfrak{g}.
    \end{align*}
\end{exam}

\begin{exam}
    (Embedding tensors) Let $( \mathfrak{g}, [~, ~]_\mathfrak{g})$ be a Lie algebra and $(V, \rho)$ be a representation. Then a linear map $r : V \rightarrow \mathfrak{g}$ is a deformation map in the hemi-semidirect product $( \mathfrak{g} \oplus V, [~,~]_\mathrm{hemi})$ if and only if it satisfies
    \begin{align*}
        [r (u), r(v)]_\mathfrak{g} = r ( \rho (u) v), \text{ for all } u, v \in V,
    \end{align*}
    which means that $r$ is an {\em embedding tensor} \cite{kotov,sheng-embed}.
\end{exam}

\begin{exam}\label{agore-exam}
    (Deformation maps in a matched pair of Leibniz algebras) In \cite{agore} Agore and Militaru defined deformation maps in a matched pair of Leibniz algebras in the study of classifying compliments. 
    Let $(\mathcal{G} = \mathfrak{g} \oplus \mathfrak{h}, [~,~]_\mathcal{G})$ be a matched pair of Leibniz algebras (or a twilled Leibniz algebra), i.e. $\mathfrak{g}, \mathfrak{h} \subset \mathcal{G}$ are both Leibniz subalgebras. Thus, in terms of the components of the bracket $[~,~]_\mathcal{G}$, we have $\theta = 0$ and $\eta = 0$. Recall from \cite{agore} that a linear map $r : \mathfrak{h} \rightarrow \mathfrak{g}$ is a {\em deformation map} in the matched pair of Leibniz algebras if it satisfies
    \begin{align*}
    [ r (u), r (v)]_\mathfrak{g} + \psi^R (r (u), v) + \psi^L ( u, r (v))  = r \big( [u, v]_\mathfrak{h} + \rho^L ( r (u), v) + \rho^R (u, r (v))   \big), \text{ for all } u, v \in \mathfrak{h}.
\end{align*}
    Thus, $r$ is a deformation map in the corresponding proto-twilled Leibniz algebra.
\end{exam}

\medskip

\section{Cohomology of a deformation map}\label{sec4}
In this section, we first show that a deformation map $r: \mathfrak{h} \rightarrow \mathfrak{g}$ in a proto-twilled Leibniz algebra $(\mathcal{G} = \mathfrak{g} \oplus \mathfrak{h}, [~,~]_\mathcal{G})$ induces a Leibniz algebra structure on the vector space $\mathfrak{h}$. Moreover, there is a suitable representation of this induced Leibniz algebra. Finally, we define the cohomology of the deformation map $r$ as the Loday-Pirashvili cohomology of the induced Leibniz algebra with coefficients in that suitable representation.

\medskip

Let $(\mathcal{G} = \mathfrak{g} \oplus \mathfrak{h}, [~,~]_\mathcal{G})$ be a proto-twilled Leibniz algebra. As before, for any $(x , u) , (y, v) \in \mathfrak{g} \oplus \mathfrak{h} $, we let
\begin{align*}
     [ (x , u) , (y, v) ]_\mathcal{G} = \big(  [x, y]_\mathfrak{g} + \psi^R (x, v) + \psi^L (u, y) + \eta (u, v) ~ \! , ~ \! [u, v]_\mathfrak{h}  + \rho^L (x, v) + \rho^R (u, y) + \theta (x, y) \big).
\end{align*}

\begin{prop}
    A linear map $r : \mathfrak{h} \rightarrow \mathfrak{g}$ is a deformation map if and only if the graph $\mathrm{Gr} (r) = \{ ( r(u), u) | ~ \! u \in \mathfrak{h} \}$ is a subalgebra of the Leibniz algebra $(\mathcal{G}, [~,~]_\mathcal{G})$.
\end{prop}

\begin{proof}
    Let $(r(u), u) ,  (r(v), v) \in \mathrm{Gr} (r)$ be two arbitrary elements. Then we have
    \begin{align*}
        [ (r(u), u), (r (v), v)]_\mathcal{G} =~& \big(  [ r (u), r (v)]_\mathfrak{g} + \psi^R (r (u), v) + \psi^L ( u, r (v)) + \eta (u, v), \\
        & \qquad [u, v]_\mathfrak{h} + \rho^L ( r (u), v) + \rho^R (u, r (v)) + \theta ( r (u), r (v))   \big).
    \end{align*}
    This is in $\mathrm{Gr} (r)$ if and only if the identity (\ref{defor-map-iden}) holds. Hence the result follows.
\end{proof}

Since the vector space $\mathfrak{h}$ is linearly isomorphic to $\mathrm{Gr} (r)$ via the canonical identification $u \leftrightsquigarrow (r(u), u)$, we obtain the following result.

\begin{prop}\label{ind-leib}
    Let $r : \mathfrak{h} \rightarrow \mathfrak{g}$ be a deformation map in the proto-twilled Leibniz algebra $(\mathcal{G} = \mathfrak{g} \oplus \mathfrak{h}, [~,~]_\mathcal{G})$. Then the vector space $\mathfrak{h}$ inherits a Leibniz algebra structure with the bracket $[~,~]_r : \mathfrak{h} \times \mathfrak{h} \rightarrow \mathfrak{h}$ given by
    \begin{align*}
        [u, v]_r := [u, v]_\mathfrak{h} + \rho^L ( r(u), v) + \rho^R (u, r(v)) + \theta (r (u), r(v)), \text{ for } u , v \in \mathfrak{h}.
    \end{align*}
    (We denote the Leibniz algebra $(\mathfrak{h}, [~,~]_r)$ simply by $\mathfrak{h}_r$).
\end{prop}

Let the proto-twilled Leibniz algebra structure $[~,~]_\mathcal{G}$ be given by (in terms of bidegree components)
\begin{align*}
    [~,~]_\mathcal{G} = \widetilde{\theta} + \mu + \nu + \widetilde{\eta},
\end{align*}
where $\mu$ and $\nu$ are described in (\ref{mu-nu1}) and (\ref{mu-nu2}).
Let $r: \mathfrak{h} \rightarrow \mathfrak{g}$ be any linear map (not necessarily a deformation map). Then it has been shown in \cite[Proposition 3.19]{tang-sheng} that the element
\begin{align*}
    \Omega_r = \theta_r + \mu_r + \nu_r + \eta_r \in C^{2|-1} \oplus C^{1|0} \oplus C^{0|1} \oplus C^{-1|2} = \mathrm{Hom} (\mathcal{G}^{\otimes 2}, \mathcal{G})
\end{align*}
defines a new proto-twilled Leibniz algebra structure, where
\begin{align}
    \theta_r =~& \widetilde{\theta}, \\
    \mu_r =~& \mu + \llbracket \widetilde{\theta}, \widetilde{r} \rrbracket_\mathsf{B},\\
    \nu_r =~& \nu + \llbracket \mu, \widetilde{r} \rrbracket_\mathsf{B} + \frac{1}{2} \llbracket \llbracket \widetilde{\theta} , \widetilde{r} \rrbracket_\mathsf{B}, \widetilde{r} \rrbracket_\mathsf{B}, \label{define-nu-r}\\
    \eta_r =~& \widetilde{\eta} + \llbracket \nu, \widetilde{r} \rrbracket_\mathsf{B} + \frac{1}{2} \llbracket \llbracket \mu, \widetilde{r} \rrbracket_\mathsf{B}, \widetilde{r} \rrbracket_\mathsf{B} + \frac{1}{6} \llbracket \llbracket \llbracket \widetilde{\theta}, \widetilde{r} \rrbracket_\mathsf{B}, \widetilde{r} \rrbracket_\mathsf{B}, \widetilde{r} \rrbracket_\mathsf{B}.
\end{align}
Here $\widetilde{r} \in \mathrm{Hom} (\mathcal{G}, \mathcal{G})$ is the horizontal lift of $r$ given by $\widetilde{r} (x, u) = (r(u), 0)$, for $(x, u) \in \mathcal{G} = \mathfrak{g} \oplus \mathfrak{h}$. As a consequence, we get the following result.

\begin{thm}
Let $( \mathcal{G} = \mathfrak{g} \oplus \mathfrak{h}, [~,~]_\mathcal{G})$ be a proto-twilled Leibniz algebra and $r : \mathfrak{h} \rightarrow \mathfrak{g}$ be a deformation map. Then $(\mathcal{G} = \mathfrak{g} \oplus \mathfrak{h} , [~,~]_r)$ is a quasi-twilled Leibniz algebra, where for $x , y \in \mathfrak{g}$ and $u, v \in \mathfrak{h}$,
\begin{align*}
    [ (x, 0), (y, 0)]_r :=~& \big(  [x,y]_\mathfrak{g} - r (\theta (x, y)) ~ \! , ~ \! \theta (x, y) \big),\\
    [(0, u), (0, v)]_r :=~& \big( 0, [u, v]_\mathfrak{h} + \rho^L ( r (u), v) + \rho^R (u, r(v)) + \theta ( r(u), r (v))   \big), \\
    [(x, 0), (0, u)]_r :=~& \big( \psi^R (x, u) + [ x, r(u)]_\mathfrak{g} - r ( \rho^L (x, u) + \theta (x, r (u))) ~ \! , ~ \! \rho^L (x, u) + \theta (x, r (u))   \big),\\
    [(0, u), (x, 0)]_r :=~& \big( \psi^L (u, x) + [r (u), x]_\mathfrak{g} - r ( \rho^R (u, x) + \theta ( r(u), x)) ~ \! , ~ \! \rho^R (u, x) + \theta ( r(u), x)  \big).
\end{align*}
\end{thm}

\begin{proof}
It follows from the above discussions that $( \mathcal{G} = \mathfrak{g} \oplus \mathfrak{h}, \Omega_r = [~,~]_r)$ is a proto-twilled Leibniz algebra, for any linear map $r : \mathfrak{h} \rightarrow \mathfrak{g}$. However, by a direct calculation, we get that
\begin{align*}
    \eta_r ((0, u) , (0, v)) =~& \big( \widetilde{\eta} + \llbracket \nu, \widetilde{r} \rrbracket_\mathsf{B} + \frac{1}{2} \llbracket \llbracket \mu, \widetilde{r} \rrbracket_\mathsf{B}, \widetilde{r} \rrbracket_\mathsf{B} + \frac{1}{6} \llbracket \llbracket \llbracket \widetilde{\theta}, \widetilde{r} \rrbracket_\mathsf{B}, \widetilde{r} \rrbracket_\mathsf{B}, \widetilde{r} \rrbracket_\mathsf{B} \big) ((0, u), (0, v)) \\
    =~& \big( \eta (u, v)  - r ([u, v]_\mathfrak{h}) + \psi^R (r (u), v) + \psi^L (u, r (v)) \\
     & \quad +  [r (u), r(v)]_\mathfrak{g} - r (    \rho^L (r(u), v) + \rho^R (u, r (v)) ) - r \big(  \theta (r (u), r(v))  \big) ~ \! , ~ \! 0  \big),
\end{align*}
for $u, v \in \mathfrak{h}$. Thus, $\eta_r$ vanishes identically if $r$ is a deformation map. Hence, in this case, we have
\begin{align*}
    \Omega_r = \widetilde{\theta} + \mu_r + \nu_r \in C^{2 |-1} \oplus C^{1|0} \oplus C^{0|1}
\end{align*}
which equivalently says that $(\mathfrak{g} \oplus \mathfrak{h}, [~,~]_r)$ is a quasi-twilled Leibniz algebra. In particular, we have
\begin{align*}
    [(x, 0), (y, 0)]_r = \Omega_r ((x, 0), (y, 0)) =~& (\widetilde{\theta} + \mu_r) ((x, 0), (y, 0)) \\
    =~& ( \widetilde{\theta} + \mu + \llbracket \widetilde{\theta}, \widetilde{r} \rrbracket_\mathsf{B} ) ((x, 0), (y, 0)) \\
    =~& \big(  [x, y]_\mathfrak{g} - r ( \theta (x, y)) ~ \! , ~ \! \theta (x, y)   \big),
\end{align*}
\begin{align*}
    [ (0, u), (0, v) ]_r = \Omega_r ((0, u), (0, v)) =~& \nu_r ((0, u), (0, v)) \\
    =~& \big( \nu + \llbracket \mu, \widetilde{r} \rrbracket_\mathsf{B} + \frac{1}{2} \llbracket \llbracket \widetilde{\theta} , \widetilde{r} \rrbracket_\mathsf{B}, \widetilde{r} \rrbracket_\mathsf{B} \big)  ((0, u), (0, v)) \\
    =~& (0 ~ \! ,  [u, v]_\mathfrak{h} + \rho^L ( r (u), v) + \rho^R (u, r(v)) + \theta ( r(u), r (v)) ),
\end{align*}
\begin{align*}
    [(x, 0), (0, u)]_r =~& \Omega_r ((x, 0), (0, u)) \\ =~& (\mu_r + \nu_r) ((x, 0), (0, u)) \\
    =~& \big( \mu + \nu + \llbracket \mu + \widetilde{\theta} , \widetilde{r} \rrbracket_\mathsf{B} + \frac{1}{2} \llbracket \llbracket \widetilde{\theta}, \widetilde{r} \rrbracket_\mathsf{B}, \widetilde{r} \rrbracket_\mathsf{B}   \big) ((x, 0), (0, u)) \\
    =~& \big( \psi^R (x, u) + [ x, r(u)]_\mathfrak{g} - r ( \rho^L (x, u) + \theta (x, r (u))) ~ \! , ~ \! \rho^L (x, u) + \theta (x, r (u))   \big),
\end{align*}
\begin{align*}
    [ (0, u), (x, 0)]_r =~& \Omega_r ((0, u), (x, 0)) \\
    =~& (\mu_r + \nu_r) ((0, u) , (x, 0)) \\
    =~& \big( \mu + \nu + \llbracket \mu + \widetilde{\theta} , \widetilde{r} \rrbracket_\mathsf{B} + \frac{1}{2} \llbracket \llbracket \widetilde{\theta}, \widetilde{r} \rrbracket_\mathsf{B}, \widetilde{r} \rrbracket_\mathsf{B}   \big)  ((0, u) , (x, 0)) \\
    =~& \big( \psi^L (u, x) + [r (u), x]_\mathfrak{g} - r ( \rho^R (u, x) + \theta ( r(u), x)) ~ \! , ~ \! \rho^R (u, x) + \theta ( r(u), x)  \big).
\end{align*}
This completes the proof.
\end{proof}

\begin{remark}\label{remark-ind-rep}
(i) It follows from the above theorem that the maps $\psi^L_r: \mathfrak{h} \times \mathfrak{g} \rightarrow \mathfrak{g}$ and $\psi^R_r: \mathfrak{g} \times \mathfrak{h} \rightarrow \mathfrak{g}$ defined by
\begin{align}
    \psi^L_r (u, x) :=~& \psi^L (u, x) + [ r(u), x]_\mathfrak{g} - r \big(   \rho^R (u, x) + \theta ( r(u), x) \big), \label{ind-left}\\
    \psi^R_r (x, u) :=~& \psi^R (x, u) + [x, r (u)]_\mathfrak{g} - r \big(  \rho^L (x, u) + \theta (x, r (u)) \big) \label{ind-right}
\end{align}
makes the triple $(\mathfrak{g}, \psi^L_r , \psi^R_r)$ into a representation of the induced Leibniz algebra $\mathfrak{h}_r$.

(ii) With the above notations, we have
\begin{align*}
    \nu_r ((x, u), (y, v)) =~& \nu_r \big(  (x, 0) + (0, u) , (y,0) + (0, v)  \big) \\
    =~& \nu_r ((x, 0), (0, v)) + \nu_r ((0, u), (y, 0)) + \nu_r ((0, u), (0, v)) \\
    =~& \big( \psi^R_r (x, v) + \psi^L_r (u, y) ~ \! , ~ \! [u, v]_r   \big)
\end{align*}
which implies that $\nu_r = \widetilde{[~,~]_r} + \widetilde{\psi^L_r} + \widetilde{\psi^R_r}$.
\end{remark}

Let $(\mathcal{G} = \mathfrak{g} \oplus \mathfrak{h}, [~,~]_\mathcal{G})$ be a proto-twilled Leibniz algebra and $r : \mathfrak{h} \rightarrow \mathfrak{g}$ be a deformation map. Then we have seen in Proposition \ref{ind-leib} that $ \mathfrak{h}_r = (\mathfrak{h}, [~, ~]_r)$ is a Leibniz algebra. Moreover, Remark \ref{remark-ind-rep}(i) says that the triple $(\mathfrak{g}, \psi^L_r, \psi^R_r)$ is a representation of the Leibniz algebra $\mathfrak{h}_r.$ Hence we may consider the Loday-Pirashvili cochain complex of the Leibniz algebra $\mathfrak{h}_r$ with coefficients in the above representation. Explicitly, for each $n \geq 0$, we let $C^n (r) := \mathrm{Hom} (\mathfrak{h}^{\otimes n}, \mathfrak{g})$ and define a map $\delta^r_\mathrm{Leib} : C^n (r) \rightarrow C^{n+1} (r)$ by
\begin{align}
    \delta^r_\mathrm{Leib} (f) (u_1, \ldots, u_{n+1}) =~& \sum_{i=1}^n (-1)^{i+1} ~ \! \psi^L_r (u_i, f ( u_1, \ldots, \widehat{ u_i}, \ldots, u_{n+1} ) ) + (-1)^{n+1} ~ \! \psi_r^R ( f (u_1, \ldots, u_n), u_{n+1} ) \nonumber \\
    ~& + \sum_{ 1 \leq i < j \leq n+1} (-1)^i ~ \! f ( u_1, \ldots, \widehat{ u_i} , \ldots, u_{j-1}, [u_i, u_j]_r, u_{j+1}, \ldots, u_{n+1}), \label{delta-r}
\end{align}
for $f \in C^n (r)$ and $u_1, \ldots, u_{n+1} \in \mathfrak{h}$.
Then we have $(\delta^r_\mathrm{Leib})^2 = 0$. The cohomology groups of the cochain complex $\{ C^\bullet (r), \delta^r_\mathrm{Leib} \}$ are called the {\bf cohomology groups of the deformation map $r$}.

\medskip

    The cohomologies of Leibniz algebra homomorphisms \cite{mandal}, derivations \cite{das-leib-hrs}, crossed homomorphisms \cite{li-wang}, Rota-Baxter operators of any weight \cite{tang-sheng-zhou} \cite{das-leib-pub}, modified Rota-Baxter operators \cite{saha}, twisted Rota-Baxter operators \cite{das-guo}, Reynolds operators \cite{das-guo} on Leibniz algebras and embedding tensors \cite{sheng-embed} on Lie algebras are extensively studied in last few years. The cohomologies for all these operators are defined by the Loday-Pirashvili cohomology of the induced Leibniz algebras (induced by the respective operators) with coefficients in suitable representation. We have already observed that all these operators can be seen as deformation maps in appropriate proto-twilled Leibniz algebras. Our cohomology of a deformation map in a general proto-twilled Leibniz algebra unifies the existing cohomologies of all these operators. 
    
    (i) In particular, as a new result, we obtain the cohomology of a deformation map in a given matched pair of Leibniz algebras. More precisely, let $(\mathcal{G} = \mathfrak{g} \oplus \mathfrak{h}, [~,~]_\mathcal{G})$ be a matched pair of Leibniz algebras and $r : \mathfrak{h} \rightarrow \mathfrak{g}$ be a deformation map (cf. Example \ref{agore-exam}). Then the vector space $\mathfrak{h}$ inherits a Leibniz algebra structure with the bracket
    \begin{align*}
        [u, v]_r := [u, v]_\mathfrak{h} + \rho^L ( r(u), v) + \rho^R (u, r(v)), \text{ for } u, v \in \mathfrak{h}.
    \end{align*}
    Moreover, the Leibniz algebra $\mathfrak{h}_r = (\mathfrak{h}, [~,~]_r)$ has a representation on the vector space $\mathfrak{g}$ with the left and right action maps $\psi^L_r : \mathfrak{h} \times \mathfrak{g} \rightarrow \mathfrak{g}$ and $\psi^R_r : \mathfrak{g} \times \mathfrak{h} \rightarrow \mathfrak{g}$ that are respectively given by
    \begin{align*}
        \psi^L_r (u, x) = \psi^L (u, x) + [r (u), x]_\mathfrak{g} - r (\rho^R (u, x)) ~~~\text{ and } ~~~ \psi^R_r (x, u) = \psi^R (x, u) + [x, r (u)]_\mathfrak{g} -  r (\rho^L (x, u)),
    \end{align*}
    for $u \in \mathfrak{h}$ and $x \in \mathfrak{g}$. The Loday-Pirashvili cohomology groups of the Leibniz algebra $\mathfrak{h}_r$ with coefficients in the above representation $(\mathfrak{g}, \psi^L_r, \psi^R_r)$ are defined to be the cohomology groups of the deformation map $r$ in the given matched pair of Leibniz algebras.

    (ii) Let $(\mathfrak{g}, [~,~]_\mathfrak{g})$ be a Leibniz algebra and $s$ be a classical Leibniz r-matrix. Then we have seen in Example \ref{exam-relative} that the map $s^\sharp : \mathfrak{g}^* \rightarrow \mathfrak{g}$ is a deformation map in the semidirect product Leibniz algebra $(\mathfrak{g} \oplus \mathfrak{g}^*, [~,~]_\ltimes)$. As a consequence, the vector space $\mathfrak{g}^*$ inherits a Leibniz algebra structure with the bracket
    \begin{align*}
        [\alpha, \beta]_{s^{\sharp}} := coad^L (s^{\sharp} (\alpha), \beta) + coad^R (\alpha, s^\sharp (\beta)), \text{ for } \alpha, \beta \in \mathfrak{g}^*.
    \end{align*}
    Moreover, there is a representation of the Leibniz algebra $(\mathfrak{g}^*, [~,~]_{s^\sharp})$ on the vector space $\mathfrak{g}$ with the action maps
    \begin{align*}
        \psi^L_{s^\sharp} (\alpha, x) := [s^\sharp (\alpha), x]_\mathfrak{g} - s^\sharp (coad^R (\alpha, x)) \quad \text{ and } \quad \psi^R_{s^\sharp} (x, \alpha) := [x, s^\sharp (\alpha)]_\mathfrak{g} - s^\sharp (coad^L (x, \alpha)),
    \end{align*}
    for $\alpha \in \mathfrak{g}^*$ and $x \in \mathfrak{g}$. The corresponding Loday-Pirashvili cohomology is defined to be cohomology of the classical Leibniz r-matrix $s$.

    \medskip

\section{Maurer-Cartan characterization and the governing algebra of a deformation map}\label{sec5}

Given a proto-twilled Leibniz algebra, here we will construct a curved $L_\infty$-algebra whose Maurer-Cartan elements are precisely deformation maps. This curved $L_\infty$-algebra is called the {\em controlling algebra} for deformation maps. In particular, we get the Maurer-Cartan characterizations of modified Rota-Baxter operators, twisted Rota-Baxter operators and Reynolds operators on a Leibniz algebra. Finally, given a fixed deformation map $r$, we will construct the {\em governing algebra} that governs the linear deformations of the operator $r$.

Let $(\mathcal{G} = \mathfrak{g} \oplus \mathfrak{h}, [~,~]_\mathcal{G})$ be a proto-twilled Leibniz algebra. Consider the Balavoine's graded Lie algebra 
\begin{align*}
    \mathfrak{B} = \big(   \oplus_{n=0}^\infty \mathrm{Hom} (  \mathcal{G}^{\otimes n+1}, \mathcal{G}   ), \llbracket ~,~ \rrbracket_\mathsf{B}  \big)
\end{align*}
associated to the vector space $\mathcal{G} = \mathfrak{g} \oplus \mathfrak{h}$. Note that the Leibniz bracket $[~,~]_\mathcal{G}$ corresponds to an element 
\begin{align*}
\Omega = \widetilde{\theta} + \mu + \nu + \widetilde{\eta} ~~ \text{ satisfying } ~~  \llbracket \Omega, \Omega \rrbracket_\mathsf{B} = 0, \text{where } \mu  \text{ and } \nu \text{  are respectively given in } (\ref{mu-nu1}) \text{ and } (\ref{mu-nu2}).
\end{align*}
Let $\mathfrak{a} = \oplus_{n=0}^\infty \mathrm{Hom} (\mathfrak{h}^{\otimes n+1}, \mathfrak{g}) = \oplus_{n=0}^\infty C^{ -1|n+1}$. Then it follows from Lemma \ref{lemma-tang} that $\mathfrak{a} \subset \mathfrak{B}$ is an abelian graded Lie subalgebra. If $P: \mathfrak{B} \rightarrow \mathfrak{a}$ is the projection onto the subspace $\mathfrak{a}$ then we have
\begin{align*}
    (\mathrm{ker ~ \!} P)_n = C^{n+1 | -1} \oplus C^{n|0} \oplus \cdots \oplus C^{0|n}  = \mathcal{Q}_n.
\end{align*}
Hence $\mathrm{ker ~ \! } P = \oplus_{n=0}^\infty (\mathrm{ker ~ \! } P)_n = \oplus_{n=0}^\infty \mathcal{Q}_n$ is a graded Lie subalgebra 
of $\mathfrak{B}$ (see Proposition \ref{prop-mqr}). Thus, the quadruple $(\mathfrak{B}, \mathfrak{a}, P, \Omega)$ is a curved $V$-data. 

\begin{thm}\label{thm-mc-operator}
    Let $(\mathcal{G} = \mathfrak{g} \oplus \mathfrak{h}, [~, ~]_\mathcal{G})$ be a proto-twilled Leibniz algebra.
    \begin{itemize}
        \item[(i)] Then the pair $\big(  \mathfrak{a} = \oplus_{n=0}^\infty \mathrm{Hom} ( \mathfrak{h}^{\otimes n+1}, \mathfrak{g}), \{ l_k \}_{k=0}^\infty \big)$ is a curved $L_\infty$-algebra, where for $f, g, h \in \mathfrak{a}$,
        \begin{align*}
            l_0 =~& \eta, \\
            l_1 (f) =~& \llbracket \nu, f \rrbracket_\mathsf{B}, \\
            l_2 (f, g) =~& \llbracket \llbracket \mu, f \rrbracket_\mathsf{B} , g \rrbracket_\mathsf{B},\\
            l_3 (f,g,h) =~& \llbracket \llbracket \llbracket \widetilde{\theta}, f \rrbracket_\mathsf{B} , g \rrbracket_\mathsf{B} , h \rrbracket_\mathsf{B},\\
            l_k =~& 0, \text{ for } k \geq 4.
        \end{align*}
        \item[(ii)] A linear map $r: \mathfrak{h} \rightarrow \mathfrak{g}$ is a deformation map in the given proto-twilled Leibniz algebra if and only if $r \in \mathfrak{a}_0 = \mathrm{Hom} (\mathfrak{h}, \mathfrak{g})$ is a Maurer-Cartan element of the above curved $L_\infty$-algebra $(\mathfrak{a}, \{ l_k \}_{k=0}^\infty)$.
    \end{itemize}
\end{thm}

\begin{proof}
(i) Since $(\mathfrak{B}, \mathfrak{a}, P, \Omega)$ is a curved $V$-data, it follows from Theorem \ref{thm-voro} that $(\mathfrak{a}, \{ l_k \}_{k=0}^\infty )$ is a curved $L_\infty$-algebra. The structure maps are precisely given by
\begin{align}
    l_0 =~& P (\Omega ) = \eta, \label{notation1}\\
    l_1 (f) =~& P \llbracket \Omega, f \rrbracket_\mathsf{B} = \llbracket \nu, f \rrbracket_\mathsf{B}, \\
    l_2 (f, g) =~& P \llbracket \llbracket \Omega, f \rrbracket_\mathsf{B} , g \rrbracket_\mathsf{B} = \llbracket \llbracket \mu, f \rrbracket_\mathsf{B} , g \rrbracket_\mathsf{B},\\
    l_3 (f, g, h ) =~&  P \llbracket \llbracket \llbracket \Omega, f \rrbracket_\mathsf{B} , g \rrbracket_\mathsf{B},  h \rrbracket_\mathsf{B} =  \llbracket \llbracket \llbracket \overline{\theta}, f \rrbracket_\mathsf{B} , g \rrbracket_\mathsf{B},  h \rrbracket_\mathsf{B}, \label{notation4}
\end{align}
and for any $k \geq 4$ and $f_1, \ldots, f_k \in \mathfrak{a}$,
\begin{align*}
    l_k (f_1, \ldots, f_k) = P \llbracket \cdots \llbracket \llbracket \Omega, f_1 \rrbracket_\mathsf{B}, f_2 \rrbracket_\mathsf{B}, \ldots, f_k \rrbracket_\mathsf{B} = 0 \quad(\text{follows from Lemma }\ref{lemma-tang}).
\end{align*}

    (ii) For any linear map $r \in \mathfrak{a}_0 = \mathrm{Hom}( \mathfrak{h} ,\mathfrak{g})$, we have
    \begin{align*}
        &\big( l_0 + \sum_{k=1}^\infty \frac{1}{k!} ~ \! l_k (r, \ldots, r)   \big) (u, v) \\
        &= l_0 (u, v) + l_1 (r) (u, v) + \frac{1}{2} l_2 (r,r) (u, v) + \frac{1}{6} l_3 (r, r, r) (u, v) \\
        &= \eta (u, v) + \llbracket \nu, r \rrbracket_\mathsf{B} (u, v) + \frac{1}{2} \llbracket \llbracket \mu , r \rrbracket_\mathsf{B}, r \rrbracket_\mathsf{B} (u, v) + \frac{1}{6} \llbracket \llbracket \llbracket \widetilde{\theta}, r \rrbracket_\mathsf{B}, r \rrbracket_\mathsf{B}, r \rrbracket_\mathsf{B} (u, v) \\
        &= \eta (u, v) + \big\{ - r ([u, v]_\mathfrak{h}) + \psi^R (r (u), v) + \psi^L (u, r (v))  \big\} \\
        & \qquad \quad + \big\{ [r (u), r(v)]_\mathfrak{g} - r (    \rho^L (r(u), v) + \rho^R (u, r (v)) )    \big\} - r \big(  \theta (r (u), r(v))  \big).
    \end{align*}
    This shows that $r$ is a Maurer-Cartan element of the curved $L_\infty$-algebra $(\mathfrak{a}, \{ l_k \}_{k=0}^\infty )$ if and only if $r$ is a deformation map in the given proto-twilled Leibniz algebra.
\end{proof}

Since the Maurer-Cartan elements of the curved $L_\infty$-algebra $\big(   \mathfrak{a} = \oplus_{n=0}^\infty \mathrm{Hom} ( \mathfrak{h}^{\otimes n+1}, \mathfrak{g}), \{ l_k \}_{k=0}^\infty    \big)$ are precisely deformation maps in the given proto-twilled Leibniz algebra, we call this curved $L_\infty$-algebra as the {\bf controlling algebra} for deformation maps.

By considering the proto-twilled Leibniz algebras given in Section \ref{sec3} and finding the controlling algebras for deformation maps on them, we get/recover the Maurer-Cartan characterizations of various well-known operators on Leibniz algebras.

\begin{prop}
    Let $(\mathfrak{g}, [~,~]_\mathfrak{g})$ and $(\mathfrak{h}, [~,~]_\mathfrak{h})$ be two Leibniz algebras. Then $\big(  \oplus_{n=0}^\infty \mathrm{Hom} ( \mathfrak{h}^{\otimes n+1}, \mathfrak{g}), \{ l_k \}_{k=1}^\infty \big)$ is an $L_\infty$-algebra, where $l_k = 0$ for $k \neq 1, 2$ and
    \begin{align*}
        l_1 (f) = \llbracket  \widetilde{ ~\! [~,~]_\mathfrak{h}} ~\! , ~\! f
        \rrbracket_\mathsf{B} ~~~ \text{ and } ~~~
        l_2 (f, g) = \llbracket \llbracket \widetilde{ ~\! [~,~]_\mathfrak{g}} ~\! , ~\!  f \rrbracket_\mathsf{B}, g \rrbracket_\mathsf{B},
    \end{align*}
    for $f \in \mathrm{Hom} (\mathfrak{h}^{\otimes m} , \mathfrak{g})$ and $g \in \mathrm{Hom} (\mathfrak{h}^{\otimes n} , \mathfrak{g})$. Moreover, a linear map $\varphi : \mathfrak{h} \rightarrow \mathfrak{g}$ is a Leibniz algebra homomorphism if and only if $\varphi \in \mathrm{Hom} (\mathfrak{h} , \mathfrak{g})$ is a Maurer-Cartan element of the above $L_\infty$-algebra.
\end{prop}

\begin{prop}
    Let $(\mathfrak{g}, [~, ~]_\mathfrak{g})$ be a Leibniz algebra and $(V, \rho^L, \rho^R)$ be a representation of it. Then the pair $\big(  \oplus_{n=0}^\infty \mathrm{Hom} (\mathfrak{g}^{\otimes n+1}, V) , \{ l_k \}_{k=1}^\infty  \big)$ is an $L_\infty$-algebra, where $l_k = 0$ for $k \neq 1$ and
    \begin{align*}
        l_1 (f) (x_1, \ldots, x_{n+1}) =~& \sum_{i=1}^n (-1)^{i+n} ~ \! \rho^L ( x_i ,  f (x_1, \ldots, \widehat{ x_i} , \ldots, x_{n+1}) ) +  \rho^R (f (x_1, \ldots, x_n), x_{n+1} ) \\
   &+  \sum_{1 \leq i < j \leq n+1} (-1)^{i+n+1} ~ \! f (x_1, \ldots, \widehat{x_i}, \ldots, x_{j-1}, [x_i, x_j]_\mathfrak{g}, x_{j+1}, \ldots, x_{n+1}),
    \end{align*}
    for $f \in \mathrm{Hom} (\mathfrak{g}^{\otimes n}, V)$ and $x_1, \ldots, x_{n+1} \in \mathfrak{g}$.
Moreover, a linear map $D: \mathfrak{g} \rightarrow V$ is a derivation if and only if $D \in \mathrm{Hom} (\mathfrak{g}, V)$ is a Maurer-Cartan element.
\end{prop}

It follows from the above proposition that the map $l_1$ makes the graded space $\oplus_{n=1}^\infty \mathrm{Hom} (\mathfrak{g}^{\otimes n}, V) $ into a cochain complex. Note that this cochain complex is isomorphic to the Loday-Pirashvili cochain complex of the Leibniz algebra $(\mathfrak{g}, [~,~]_\mathfrak{g})$ with coefficients in the representation $(V, \rho^L, \rho^R)$.

\begin{prop}
     Let $(\mathfrak{g}, [~, ~]_\mathfrak{g})$ be a Leibniz algebra and $(V, \rho^L, \rho^R)$ be a representation of it. Then the pair $\big(  \oplus_{n=0}^\infty \mathrm{Hom} (V^{\otimes n+1}, \mathfrak{g}), \{ l_k \}_{k=1}^\infty  \big)$ is an $L_\infty$-algebra, where $l_k = 0$ for $k \neq 2$ and 
     \begin{align*}
         l_2 (f, g) = \llbracket \llbracket  \mu, f \rrbracket_\mathsf{B} , g \rrbracket_\mathsf{B},
     \end{align*}
     for $ f \in \mathrm{Hom} (V^{\otimes m}, \mathfrak{g})$ and $g \in \mathrm{Hom} (V^{\otimes n}, \mathfrak{g})$. A linear map $r: V \rightarrow \mathfrak{g}$ is a relative Rota-Baxter operator of weight $0$ if and only if $r \in \mathrm{Hom} (V, \mathfrak{g})$ is a Maurer-Cartan element.
\end{prop}

As a remark of the above proposition, we get that the shifted space $\oplus_{n=1}^\infty \mathrm{Hom} (V^{\otimes n}, \mathfrak{g})$ endowed with the bracket $\{ \! [ f, g ] \! \} = (-1)^m ~ \! l_2 (f, g) = (-1)^m \llbracket \llbracket  \mu, f \rrbracket_\mathsf{B}, g \rrbracket_\mathsf{B},$ for $ f \in \mathrm{Hom} (V^{\otimes m}, \mathfrak{g})$ and $g \in \mathrm{Hom} (V^{\otimes n}, \mathfrak{g})$, is a graded Lie algebra \cite{tang-sheng}. See \cite{tang-sheng-zhou} for the explicit description of this graded Lie bracket.

\begin{prop}
    Let $(\mathfrak{g}, [~,~]_\mathfrak{g})$ be a Leibniz algebra. Then $\big( \oplus_{n=0}^\infty \mathrm{Hom} ( \mathfrak{g}^{\otimes n+1}, \mathfrak{g}), \{ l_k \}_{k=0}^\infty  \big)$ is a curved $L_\infty$-algebra, where $l_k = 0$ for $k \neq 0, 2$ and $l_0 = [~,~]_\mathfrak{g}$ (the given Leibniz multiplication on $\mathfrak{g}$) and
    \begin{align*}
        l_2 (f, g) = \llbracket \llbracket \mu, f \rrbracket_\mathsf{B}, g \rrbracket_\mathsf{B},
    \end{align*}
    where $\mu^{1 | 0}$ is the element given by $\mu ((x', x) , (y', y) ) = ( [x', y']_\mathfrak{g} ~ \! , ~ \! [x', y]_\mathfrak{g} + [x, y']_\mathfrak{g})$, for $(x', x) , (y', y) \in \mathfrak{g} \oplus \mathfrak{g}$.
     Moreover, a linear map $r : \mathfrak{g} \rightarrow \mathfrak{g}$ is a modified Rota-Baxter operator on the Leibniz algebra $(\mathfrak{g}, [~,~]_\mathfrak{g})$ if and only if $r \in \mathrm{Hom} (\mathfrak{g}, \mathfrak{g})$ is a Maurer-Cartan element of this curved $L_\infty$-algebra.
\end{prop}

\begin{prop}
    Let $(\mathfrak{g}, [~, ~]_\mathfrak{g})$ be a Leibniz algebra and $(V, \rho^L, \rho^R)$ be a representation of it. For any Leibniz $2$-cocycle $\theta \in Z^2_\mathrm{Leib} (\mathfrak{g}, V)$, the pair $\big(  \oplus_{n=0}^\infty \mathrm{Hom} (V^{\otimes n+1}, \mathfrak{g}), \{ l_k \}_{k=1}^\infty    \big)$ is an $L_\infty$-algebra, where $l_k = 0$ for $k \neq 2, 3$ and
    \begin{align*}
     l_2 (f, g) = \llbracket \llbracket \mu, f \rrbracket_\mathsf{B} , g \rrbracket_\mathsf{B} ~~~ \text{ and } ~~~
            l_3 (f,g,h) = \llbracket \llbracket \llbracket \widetilde{\theta}, f \rrbracket_\mathsf{B} , g \rrbracket_\mathsf{B} , h \rrbracket_\mathsf{B}.
    \end{align*}
    Moreover, a linear map $r: V \rightarrow \mathfrak{g}$ is a $\theta$-twisted Rota-Baxter operator if and only if $r \in \mathrm{Hom} (V, \mathfrak{g})$ is a Maurer-Cartan element of the above $L_\infty$-algebra.
\end{prop}

In the above proposition, if we consider $V = \mathfrak{g}$ (the adjoint representation) and take $\theta (x, y) = - [x, y]_\mathfrak{g}$ for $x, y \in \mathfrak{g}$, we obtain an $L_\infty$-algebra $\big(  \oplus_{n=0}^\infty \mathrm{Hom} ( \mathfrak{g}^{\otimes n+1}, \mathfrak{g}), \{ l_k \}_{k=1}^\infty \big)$ whose Maurer-Cartan elements are precisely Reynolds operators on the Leibniz algebra $(\mathfrak{g}, [~,~]_\mathfrak{g})$.


\begin{prop}
    Let $(\mathfrak{g}, [~, ~]_\mathfrak{g})$ be a Lie algebra and $(V, \rho)$ be a representation of it. Then the pair $(\oplus_{n=0}^\infty \mathrm{Hom} (V^{\otimes n+1}, \mathfrak{g}), \{ l_k \}_{k=1}^\infty)$ is an $L_\infty$-algebra, where $l_k = 0$ for $k \neq 2$ and 
    \begin{align*}
        l_2 (f, g) = \llbracket  \llbracket \mu_\mathrm{hemi} ~ \! \! , ~ \! \! f \rrbracket_\mathsf{B}, g \rrbracket_\mathsf{B},
    \end{align*}
    for $f \in \mathrm{Hom} (V^{\otimes m}, \mathfrak{g})$ and $g \in \mathrm{Hom} (V^{\otimes n}, \mathfrak{g})$. Here $\mu_\mathrm{hemi} \in C^{1|0}$ is given by   $\mu_\mathrm{hemi} ((x, u), (y, v)) = ([x, y]_\mathfrak{g}, \rho(x) v)$, for $(x, u) , (y, v) \in \mathfrak{g} \oplus V$. Moreover, a linear map $r: V \rightarrow \mathfrak{g}$ is an embedding tensor if and only if $r \in \mathrm{Hom} (V, \mathfrak{g})$ is a Maurer-Cartan element \cite{sheng-embed}.
\end{prop}

\begin{prop}\label{def-matched-prop}
    Let $(\mathcal{G} = \mathfrak{g} \oplus \mathfrak{h}, [~,~]_\mathcal{G})$ be a matched pair of Leibniz algebras. Then there is an $L_\infty$-algebra $\big( \oplus_{n=0}^\infty \mathrm{Hom} (\mathfrak{h}^{\otimes n+1}, \mathfrak{g}), \{ l_k \}_{k=1}^\infty  \big)$, where $l_k = 0$ for $k \neq 1, 2$ and 
    \begin{align*}
     l_1 (f) = \llbracket \nu, f \rrbracket_\mathsf{B} ~~~ \text{ and } ~~~ l_2 (f, g) = \llbracket \llbracket \mu, f \rrbracket_\mathsf{B}, g \rrbracket_\mathsf{B},
    \end{align*}
    for $f \in \mathrm{Hom} (\mathfrak{h}^{\otimes m}, \mathfrak{g})$ and $g \in \mathrm{Hom} (\mathfrak{h}^{\otimes n}, \mathfrak{g})$. The explicit description of $l_1$ is given by
    \begin{align*}
        l_1 (f) (u_1, \ldots, u_{m+1}) =~& \sum_{i=1}^m (-1)^{i+m} ~ \! \psi^L ( u_i ,  f (u_1, \ldots, \widehat{ u_i} , \ldots, u_{m+1}) ) + \! \psi^R (f (u_1, \ldots, u_m), u_{m+1} ) \\
   &+  \sum_{1 \leq i < j \leq m+1} (-1)^{i+m+1} ~ \! f (u_1, \ldots, \widehat{u_i}, \ldots, u_{j-1}, [u_i, u_j]_\mathfrak{h}, u_{j+1}, \ldots, u_{m+1}),
    \end{align*}
    for $u_1, \ldots, u_{m+1} \in \mathfrak{h}$, while the explicit description of $l_2$ can be found in \cite{tang-sheng-zhou}. Finally, a linear map $r: \mathfrak{h} \rightarrow \mathfrak{g}$ is a deformation map in the given matched pair of Leibniz algebras if and only if $r \in \mathrm{Hom} (\mathfrak{h}, \mathfrak{g})$ is a Maurer-Cartan element of the above $L_\infty$-algebra.
\end{prop}

Given a deformation map in a proto-twilled Leibniz algebra, we now construct a new $L_\infty$-algebra by the twisting method.

\begin{thm}
    Let $(\mathcal{G} = \mathfrak{g} \oplus \mathfrak{h}, [~,~]_\mathcal{G})$ be a proto-twilled Leibniz algebra and $r : \mathfrak{h} \rightarrow \mathfrak{g}$ be a deformation map. Then  $\big( \mathfrak{a} = \oplus_{n=0}^\infty \mathrm{Hom} ( \mathfrak{h}^{\otimes n+1}, \mathfrak{g}), \{ l_k^r \}_{k=1}^\infty \big)$ is an $L_\infty$-algebra, where for $f, g, h \in \mathfrak{a}$,
    \begin{align*}
        l_1^r (f) =~& \llbracket \nu, f \rrbracket_\mathsf{B} + \llbracket \llbracket \mu, r \rrbracket_\mathsf{B} , f\rrbracket_\mathsf{B} + \frac{1}{2} \llbracket \llbracket \llbracket \widetilde{\theta}, r \rrbracket_\mathsf{B}, r \rrbracket_\mathsf{B}, f \rrbracket_\mathsf{B},\\
        l_2^r (f, g) =~& \llbracket \llbracket \mu, f \rrbracket_\mathsf{B} , g \rrbracket_\mathsf{B} + \llbracket \llbracket \llbracket \widetilde{\theta}, r \rrbracket_\mathsf{B}, f \rrbracket_\mathsf{B}, g \rrbracket_\mathsf{B},\\
        l_3^r (f, g, h) =~& \llbracket \llbracket \llbracket \widetilde{\theta}, f \rrbracket_\mathsf{B}, g \rrbracket_\mathsf{B}, h \rrbracket_\mathsf{B},\\
        l_k^r =~& 0, \text{ for } k \geq 4.
    \end{align*}
    Moreover, for any linear map $r' : \mathfrak{h} \rightarrow \mathfrak{g}$, the sum $r + r'$ is also a deformation map if and only if $r' \in \mathfrak{a}_0$ is a Maurer-Cartan element of the $L_\infty$-algebra $(\mathfrak{a} , \{ l_k^r \}_{k=1}^\infty)$.
\end{thm}

\begin{proof}
Since $r: \mathfrak{h} \rightarrow \mathfrak{g}$ is a deformation map, it follows from the previous theorem that $r \in \mathfrak{a}_0 = \mathrm{Hom} (\mathfrak{h}, \mathfrak{g})$ is a Maurer-Cartan element of the curved $L_\infty$-algebra $(\mathfrak{a}, \{ l_k \}_{k=0}^\infty)$. Hence by Theorem \ref{mc-twist-thm}, one obtains an $L_\infty$-algebra $(\mathfrak{a}, \{ l_k^r \}_{k=0}^\infty)$, where
\begin{align*}
    l_k^r (f_1, \ldots, f_k) = \sum_{n=0}^\infty \frac{1}{n!} ~ \! l_{n+k} (r, \ldots, r, f_1, \ldots, f_k), \text{ for } k \geq 1 \text{ and } f_1, \ldots, f_k \in \mathfrak{a}.
\end{align*}
Thus, for any $f, g, h \in \mathfrak{a}$, we have
\begin{align*}
    l_1^r (f) =~& l_1 (f) + l_2 (r, f) + \frac{1}{2} l_3 (r,r, f) = \llbracket \nu, f \rrbracket_\mathsf{B} + \llbracket \llbracket \mu, r \rrbracket_\mathsf{B} , f\rrbracket_\mathsf{B} + \frac{1}{2} \llbracket \llbracket \llbracket \widetilde{\theta}, r \rrbracket_\mathsf{B}, r \rrbracket_\mathsf{B}, f \rrbracket_\mathsf{B}, \\
    l_2^r (f, g) =~& l_2 (f, g) + l_3 (r, f, g) = \llbracket \llbracket \mu, f \rrbracket_\mathsf{B} , g \rrbracket_\mathsf{B} + \llbracket \llbracket \llbracket \widetilde{\theta}, r \rrbracket_\mathsf{B}, f \rrbracket_\mathsf{B}, g \rrbracket_\mathsf{B},\\
    l_3^r (f, g, h) =~&l_3 (f, g,h ) = \llbracket \llbracket \llbracket \widetilde{\theta}, f \rrbracket_\mathsf{B}, g \rrbracket_\mathsf{B}, h \rrbracket_\mathsf{B},\\
    \text{ and } l_k^r =~& 0, \text{ for } k \geq 4.
\end{align*}
For the last part, we observe that
\begin{align*}
    &l_0 + l_1 (r+r') + \frac{1}{2!} ~ \! l_2 (r+r', r+ r') + \frac{1}{3!} ~ \! l_3 (r+r', r+r', r+r') \\
    &= l_0 + l_1 (r) + \frac{1}{2!} ~ \! l_2 (r,r) + \frac{1}{3!} ~ \! l_3 (r,r,r) \\ 
    & \qquad + \big\{ l_1 (r') + l_2 (r,r') + \frac{1}{2} l_2 (r',r') + \frac{1}{2} l_3 (r,r,r') + \frac{1}{2} l_3 (r,r',r') + \frac{1}{3!} l_3 (r',r',r')    \big\} \\
    &= l_1^r (r') + \frac{1}{2!} ~ \! l_2^r (r',r') + \frac{1}{3!} ~ \!  l_3^r (r',r',r').
\end{align*}
This shows that $r+r'$ is a Maurer-Cartan element of the curved $L_\infty$-algebra $(\mathfrak{a}, \{ l_k \}_{k=0}^\infty)$ if and only if $r'$ is a Maurer-Cartan element of the $L_\infty$-algebra $(\mathfrak{a}, \{ l_k^r \}_{k=1}^\infty)$. Hence the result follows.
\end{proof}

The $L_\infty$-algebra $\big( \mathfrak{a} = \oplus_{n=0}^\infty \mathrm{Hom} ( \mathfrak{h}^{\otimes n+1}, \mathfrak{g}), \{ l_k^r \}_{k=1}^\infty \big)$ constructed in the above theorem is called the {\bf governing algebra} of the deformation map $r$ as this $L_\infty$-algebra governs the linear deformations of the operator $r$.

Let $r: \mathfrak{h} \rightarrow \mathfrak{g}$ be a deformation map. Since the governing algebra  $\big( \mathfrak{a} = \oplus_{n=0}^\infty \mathrm{Hom} ( \mathfrak{h}^{\otimes n+1}, \mathfrak{g}), \{ l_k^r \}_{k=1}^\infty \big)$ is an $L_\infty$-algebra, it follows that $(l_1^r)^2 = 0$. In other words, the degree $1$ map $l_1^r: \mathfrak{a} \rightarrow \mathfrak{a}$ is a differential. The next result shows that $l_1^r$ coincides with the map $\delta^r_\mathrm{Leib}$ given in (\ref{delta-r}) up to some sign.

\begin{prop}
    For any $f \in \mathfrak{a}_{n-1} = \mathrm{Hom} (\mathfrak{h}^{\otimes n} , \mathfrak{g})$, we have $l_1^r (f) = (-1)^{n-1} ~ \! \delta^r_\mathrm{Leib} (f)$.
\end{prop}

\begin{proof}
We have
\begin{align*}
    l_1^r (f) =~& \llbracket \nu, f \rrbracket_\mathsf{B} + \llbracket \llbracket \mu, \widetilde{r} \rrbracket_\mathsf{B} , f \rrbracket_\mathsf{B} + \frac{1}{2} \llbracket \llbracket \llbracket \widetilde{\theta} , \widetilde{r} \rrbracket_\mathsf{B}, \widetilde{r} \rrbracket_\mathsf{B}, f \rrbracket_\mathsf{B} \\
    =~& \big\llbracket \nu + \llbracket \mu, \widetilde{r} \rrbracket_\mathsf{B} + \frac{1}{2} \llbracket \llbracket \widetilde{\theta} , \widetilde{r} \rrbracket_\mathsf{B}, \widetilde{r} \rrbracket_\mathsf{B} ~ \! , f \big\rrbracket_\mathsf{B} \\
    =~& \llbracket \nu_r , f  \rrbracket_\mathsf{B} \quad (\text{by } (\ref{define-nu-r})).
\end{align*}
Since $\nu_r = \widetilde{[~,~]_r} + \widetilde{\psi^L_r} + \widetilde{\psi^R_r}~$~ (cf. Remark \ref{remark-ind-rep}(ii)), it follows that $\llbracket \nu_r , f \rrbracket_\mathsf{B} = (-1)^{n-1} ~ \! \delta^r_\mathrm{Leib} (f)$. This completes the proof.
\end{proof}

\begin{remark}
    It follows from the above proposition that the cohomology of the deformation map $r$ can be equivalently described by the cohomology of the cochain complex $\{ C^\bullet (r), l_1^r \}$.
\end{remark}

\begin{remark}\label{remark-fial}
    A proto-twilled Lie algebra is a Lie algebra $(\mathcal{G}, [~,~]_\mathcal{G})$ whose underlying vector space has a direct sum decomposition $\mathcal{G} = \mathfrak{g} \oplus \mathfrak{h}$ into subspaces. In similar to Definition \ref{defn-defor-map}, a linear map $r : \mathfrak{h} \rightarrow \mathfrak{g}$ is a deformation map in the proto-twilled Lie algebra $(\mathcal{G} = \mathfrak{g} \oplus \mathfrak{h}, [~,~]_\mathcal{G})$ if its graph $\mathrm{Gr} (r)$ is a subalgebra of the Lie algebra $(\mathcal{G}, [~,~]_\mathcal{G})$. A deformation map $r :\mathfrak{h} \rightarrow \mathfrak{g}$ in a proto-twilled Lie algebra $(\mathcal{G} = \mathfrak{g} \oplus \mathfrak{h}, [~,~]_\mathcal{G})$ induces a Lie algebra structure on the vector space $\mathfrak{h}$ (denoted by $\mathfrak{h}_r^\mathrm{Lie}$). Moreover, there is a Lie algebra representation of $\mathfrak{h}_r^\mathrm{Lie}$ on the vector space $\mathfrak{g}$ (denoted by $\mathfrak{g}_r^\mathrm{Rep}$). Then the Chevalley-Eilenberg cohomology of the Lie algebra $\mathfrak{h}_r^\mathrm{Lie}$ with coefficients in the representation $\mathfrak{g}_r^\mathrm{Rep}$ is defined to be the cohomology of the deformation map $r$ in the given proto-twilled Lie algebra. This cohomology turns out to govern the linear deformations of the operator $r$.

    Let $(\mathcal{G} = \mathfrak{g} \oplus \mathfrak{h}, [~,~]_\mathcal{G})$ be a given proto-twilled Lie algebra and $r : \mathfrak{h} \rightarrow \mathfrak{g}$ be a deformation map on it. By viewing $(\mathcal{G} = \mathfrak{g} \oplus \mathfrak{h}, [~,~]_\mathcal{G})$ as a proto-twilled Leibniz algebra and $r : \mathfrak{h} \rightarrow \mathfrak{g}$ a deformation map on it, one may also define the cohomology of the operator $r$. This is precisely given by the Loday-Pirashvili cohomology of the Lie algebra $\mathfrak{h}_r^\mathrm{Lie}$ with coefficients in the representation $\mathfrak{g}_r^\mathrm{Rep}$. This cohomology turns out to govern the linear deformations of the operator $r$ (viewed as a deformation map in the proto-twilled Leibniz algebra). This is certainly different than the cohomology of the operator $r$ viewed as a deformation map in the proto-twilled Lie algebra.

     For example, let $(\mathfrak{g}, [~,~]_\mathfrak{g})$ be any Lie algebra. Consider the (direct product) proto-twilled Lie algebra $(\mathfrak{g} \oplus \mathfrak{g}, [~,~]_\mathrm{dir})$, where $[ (x, x'), (y, y')]_\mathrm{dir} = ([x, y]_\mathfrak{g}, [x', y']_\mathfrak{g})$ for $(x, x') , (y, y') \in \mathfrak{g} \oplus \mathfrak{g}$. Then the identity map $\mathrm{Id}_\mathfrak{g} : \mathfrak{g} \rightarrow \mathfrak{g}$ is a deformation map in the proto-twilled Lie algebra $(\mathfrak{g} \oplus \mathfrak{g}, [~,~]_\mathrm{dir})$. The cohomology of the deformation map $\mathrm{Id}_\mathfrak{g}$ is precisely the Chevalley-Eilenberg cohomology of the Lie algebra $(\mathfrak{g}, [~,~]_\mathfrak{g})$ with coefficients in the adjoint representation. On the other hand, the cohomology of the same map $\mathrm{Id}_\mathfrak{g}$ (now viewed as a deformation map in the corresponding proto-twilled Leibniz algebra) is given by the Loday-Pirashvili cohomology of the Lie algebra $(\mathfrak{g}, [~,~]_\mathfrak{g})$ with coefficients in the adjoint representation. In \cite{fial-mandal} the authors have shown that the Chevalley-Eilenberg cohomology of a Lie algebra $(\mathfrak{g}, [~,~]_\mathfrak{g})$ with coefficients in the adjoint representation is in general different than the Loday-Pirashvili cohomology of the same Lie algebra with adjoint representation. Moreover, the comparison between these two cohomologies can help to decide whether the Lie algebra $(\mathfrak{g}, [~,~]_\mathfrak{g})$ has more Leibniz deformations than just the Lie ones. As a conclusion, we get that a deformation map $r$ in a proto-twilled Lie algebra may have fewer linear deformations than the linear deformations of $r$ while viewing it as a deformation map in the corresponding proto-twilled Leibniz algebra.
\end{remark}

\medskip

\section{Governing algebra for simultaneous deformations of a proto-twilled Leibniz algebra and a deformation map}\label{sec6}
In this section, we first construct an $L_\infty$-algebra whose Maurer-Cartan elements correspond to pairs $(\Omega, r)$ consisting of a (proto-twilled) Leibniz algebra structure $\Omega$ on the direct sum $\mathcal{G} = \mathfrak{g} \oplus \mathfrak{h}$ of two vector spaces and a deformation map $r : \mathfrak{h} \rightarrow \mathfrak{g}$ on it. Using this characterization and the twisting method, we obtain the governing $L_\infty$-algebra that controls the simultaneous deformations of a given proto-twilled Leibniz algebra and a fixed deformation map.

Let $\mathfrak{g}$ and $\mathfrak{h}$ be two vector spaces. Let $\mathfrak{B} = \big(  \oplus_{n=0}^\infty \mathrm{Hom} (\mathcal{G}^{\otimes n+1}, \mathcal{G}), \llbracket ~, ~ \rrbracket_\mathsf{B} \big)$ be the Balavoine's graded Lie algebra associated to the direct sum $\mathcal{G} = \mathfrak{g} \oplus \mathfrak{h}$. Then we have already seen that $\mathfrak{a} = \oplus_{n=0}^\infty \mathrm{Hom} (\mathfrak{h}^{\otimes n+1}, \mathfrak{g})$ is an abelian graded Lie subalgebra of $\mathfrak{B}$. Further, if $P : \mathfrak{B} \rightarrow \mathfrak{a}$ is the projection map then $\mathrm{ker ~ \!} P \subset \mathfrak{B}$ is a graded Lie subalgebra. Hence $(\mathfrak{B}, \mathfrak{a}, P, \Delta = 0)$ is a $V$-data. Therefore, as a consequence of Theorem \ref{thm-voro}, we obtain the following.

\begin{thm}\label{mc-simul}
Let $\mathfrak{g}$ and $\mathfrak{h}$ be two vector spaces.
\begin{itemize}
\item[(i)] Then $(s^{-1} \mathfrak{B} \oplus \mathfrak{a}, \{ \widetilde{l}_k \}_{k=1}^\infty)$ is an $L_\infty$-algebra, where
    \begin{align*}
        \widetilde{l_1} ((s^{-1} F, f)) =~& (0, P (F)), \\
        \widetilde{l_2} ((s^{-1} F, 0), (s^{-1} G, 0)) =~& \big( (-1)^{|F|} ~ \! s^{-1} \llbracket F, G \rrbracket_\mathsf{B} , 0   \big), \\
        \widetilde{l_k} ((s^{-1} F, 0), (0, f_1), \ldots , (0, f_{k-1}) ) =~& \big( 0, P \llbracket \cdots \llbracket \llbracket F, f_1 \rrbracket_\mathsf{B}, f_2 \rrbracket_\mathsf{B}, \ldots, f_{k-1} \rrbracket_\mathsf{B} \big), ~~ k \geq 2,
    \end{align*}
    for homogeneous elements $F, G \in \mathfrak{B}$ and $f, f_1, \ldots, f_{k-1} \in \mathfrak{a}$.
    \item[(ii)] Suppose there is an element $\Omega \in \mathrm{Hom} (\mathcal{G}^{\otimes 2}, \mathcal{G})$ and a linear map $r : \mathfrak{h} \rightarrow \mathfrak{g}$. Then $\Omega$ defines a Leibniz algebra structure (i.e. $(\mathcal{G} = \mathfrak{g} \oplus \mathfrak{h}, \Omega)$ is a proto-twilled Leibniz algebra) and $r$  
    is a deformation map if and only if $\alpha = (s^{-1} \Omega , r) \in (s^{-1} \mathfrak{B} \oplus \mathfrak{a})_0$ is a Maurer-Cartan element of the above $L_\infty$-algebra $(s^{-1} \mathfrak{B} \oplus \mathfrak{a}, \{ \widetilde{l}_k \}_{k=1}^\infty)$.
    \end{itemize}
\end{thm}

\begin{proof}
    As mentioned earlier, the first part follows from Theorem \ref{thm-voro}. For the second part, we observe that
    \begin{align*}
        \llbracket \llbracket \llbracket \llbracket \Omega, r \rrbracket_\mathsf{B}, r \rrbracket_\mathsf{B}, r \rrbracket_\mathsf{B}, r \rrbracket_\mathsf{B} = 0 \quad (\text{by Lemma }\ref{lemma-tang}).
    \end{align*}
    As a consequence, $\widetilde{l_k} \big(  (s^{-1} \Omega , r), \ldots,  (s^{-1} \Omega , r) \big) = 0$ for $k \geq 5$. Hence
    \begin{align*}
        &\sum_{k=1}^\infty \frac{1}{k!} ~ \! \widetilde{l_k} \big(  (s^{-1} \Omega , r), \ldots,  (s^{-1} \Omega , r)    \big) \\
        &= \widetilde{l_1} ( (s^{-1} \Omega , r)) + \frac{1}{2!} ~ \! \widetilde{l_2} \big( (s^{-1} \Omega , r), (s^{-1} \Omega , r) \big) + \frac{1}{3!} ~ \! \widetilde{l_3} \big( (s^{-1} \Omega , r), (s^{-1} \Omega , r), (s^{-1} \Omega , r) \big) \\
        & \qquad \qquad \qquad \qquad \qquad + \frac{1}{4!} ~ \! \widetilde{l_4} \big( (s^{-1} \Omega , r), (s^{-1} \Omega , r), (s^{-1} \Omega , r), (s^{-1} \Omega , r) \big) \\
        &= (0, P (\Omega)) + \frac{1}{2} \big\{  \widetilde{l_2} \big( (s^{-1} \Omega , 0), (s^{-1} \Omega , 0) \big) + 2 ~ \! \widetilde{l_2} \big(  (s^{-1} \Omega , 0), (0, r) \big)  \big\} + \frac{1}{3!} \big\{ 3 ~ \! \widetilde{l_3} \big(  (s^{-1} \Omega , 0), (0, r), (0, r) \big)  \big\} \\
        &\qquad \qquad \qquad \qquad \qquad + \frac{1}{4!} ~ \!  \big\{  4 ~ \! \widetilde{l_4} \big(   (s^{-1} \Omega , 0), (0, r) , (0, r), (0, r) \big) \big\}  \\
        &= (0, P (\Omega)) + \frac{1}{2} \big\{     (- s^{-1} \llbracket \Omega, \Omega \rrbracket_\mathsf{B}, 0) + 2 ~ \! (0 , P \llbracket \Omega, r \rrbracket_\mathsf{B})  \big\} + \frac{1}{3!} \big\{  3 ~ \! ( 0, P \llbracket \llbracket \Omega, r \rrbracket_\mathsf{B}, r \rrbracket_\mathsf{B})  \big\} \\
        &\qquad \qquad \qquad \qquad \qquad + \frac{1}{4!} ~ \!  \big\{ 4 ~ \! (0, P \llbracket \llbracket \llbracket \Omega, r \rrbracket_\mathsf{B}, r \rrbracket_\mathsf{B}, r \rrbracket_\mathsf{B}) \big\} \\
        &= \big( - \frac{1}{2} s^{-1} \llbracket \Omega, \Omega \rrbracket_\mathsf{B} ~ \! , ~ \! l_0 + l_1 (r) + \frac{1}{2!} l_2 (r, r) + \frac{1}{3!} l_3 (r, r, r)   \big) \quad (\text{see the notations from }(\ref{notation1})\text{-}(\ref{notation4})).
    \end{align*}
    Thus, $(s^{-1} \Omega , r) \in (s^{-1} \Omega \oplus \mathfrak{a})_0$ is a Maurer-Cartan element of the $L_\infty$-algebra $(s^{-1} \mathfrak{B} \oplus \mathfrak{a} , \{ \widetilde{l_k} \}_{k=1}^\infty)$ if and only if 
    \begin{align}\label{last-eqnn}
        \llbracket \Omega, \Omega \rrbracket_\mathsf{B} = 0 \quad \text{ and } \quad l_0 + l_1 (r) + \frac{1}{2!} ~ \! l_2 (r,r) + \frac{1}{3!} ~ \! l_3 (r,r,r) = 0.
    \end{align}
    Note that $\llbracket \Omega, \Omega \rrbracket_\mathsf{B} = 0$ is equivalent that $(\mathcal{G} = \mathfrak{g} \oplus \mathfrak{h}, \Omega)$ is a Leibniz algebra while the second condition of (\ref{last-eqnn}) is equivalent that $r$ is a deformation map (cf. Theorem \ref{thm-mc-operator}). This completes the proof.
\end{proof}

Let $(\mathcal{G} = \mathfrak{g} \oplus \mathfrak{h}, \Omega)$ be a proto-twilled Leibniz algebra and $r : \mathfrak{h} \rightarrow \mathfrak{g}$ be a deformation map. Then it follows from the previous theorem that $\alpha = (s^{-1} \Omega, r)$ is a Maurer-Cartan element of the $L_\infty$-algebra $( s^{-1} \mathfrak{B} \oplus \mathfrak{a}, \{ \widetilde{l_k} \}_{k=1}^\infty)$. Hence by the twisting method (cf. Theorem \ref{mc-twist-thm}), one gets the new $L_\infty$-algebra $\big( s^{-1} \mathfrak{B} \oplus \mathfrak{a} , \{ \widetilde{l_k}^{(s^{-1} \Omega, r)} \}_{k=1}^\infty  \big)$. Moreover, we have the following.

\begin{thm}
    Let $(\mathcal{G} = \mathfrak{g} \oplus \mathfrak{h}, \Omega)$ be a proto-twilled Leibniz algebra and $r : \mathfrak{h} \rightarrow \mathfrak{g}$ be a deformation map. Then for any $\Omega' \in \mathrm{Hom} (\mathcal{G}^{\otimes 2}, \mathcal{G})$ and a linear map $r' : \mathfrak{h} \rightarrow \mathfrak{g}$, the pair $(\mathcal{G} = \mathfrak{g} \oplus \mathfrak{h}, \Omega + \Omega')$ is a (proto-twilled) Leibniz algebra and $r + r' : \mathfrak{h} \rightarrow \mathfrak{g}$ is a deformation map in it if and only if  $\alpha' = (s^{-1} \Omega', r') \in (s^{-1} \mathfrak{B} \oplus \mathfrak{a})_0$ is a Maurer-Cartan element of the $L_\infty$-algebra $(s^{-1} \mathfrak{B} \oplus \mathfrak{a}, \{  \widetilde{l}_k^{  (s^{-1} \Omega, r)}  \}_{k=1}^\infty)$.
\end{thm}

\begin{proof}
By a straightforward calculation, we get that
\begin{align*}
    \sum_{k=1}^4 \frac{1}{k!} ~ \! \widetilde{l_k} &\big(    \big( s^{-1} (\Omega + \Omega'), r+ r' \big), \ldots,  \big( s^{-1} (\Omega + \Omega'), r+ r' \big)    \big) \\
   & = \sum_{k=1}^4 \frac{1}{k!} ~ \! \widetilde{l_k}^{ (s^{-1} \Omega, r) } \big( (s^{-1} \Omega', r'), \ldots, (s^{-1} \Omega', r')   \big)
\end{align*}
which assures the result.
\end{proof}

It follows from the above theorem that the $L_\infty$-algebra $(s^{-1} \mathfrak{B} \oplus \mathfrak{a}, \{  \widetilde{l}_k^{  (s^{-1} \Omega, r)}  \}_{k=1}^\infty)$ governs the simultaneous deformations of the given proto-twilled Leibniz algebra and the deformation map.

\medskip

Given two vector spaces $\mathfrak{g}$ and $\mathfrak{h}$, in Theorem \ref{mc-simul}, we constructed the $L_\infty$-algebra $(s^{-1} \mathfrak{B} \oplus \mathfrak{a} , \{ \widetilde{l_k} \}_{k=1}^\infty)$ whose Maurer-Cartan elements are precisely given by pairs $(\Omega, r)$ consisting of a Leibniz algebra structure $\Omega$ on the direct sum $\mathcal{G} = \mathfrak{g} \oplus \mathfrak{h}$ of the vector spaces and a deformation map $r : \mathfrak{h} \rightarrow \mathfrak{g}$ on it. To obtain this $L_\infty$-algebra, we have used Theorem \ref{thm-voro} for the graded Lie subalgebra $\mathfrak{B}' \subset \mathfrak{B}$. However, we could choose some different graded Lie subalgebras $\mathfrak{B}' \subset \mathfrak{B}$ and get new $L_\infty$-algebras. For example, if we take $\mathfrak{B}' = \oplus_{n=0}^\infty C^{n|0}$ (this is a graded Lie subalgebra of $\mathfrak{B}$ by Lemma \ref{lemma-tang}) then the corresponding $L_\infty$-algebra becomes $(s^{-1} \mathfrak{B}' \oplus \mathfrak{a} , \{ \widetilde{l_k} \}_{k=1}^\infty)$. Note that a Maurer-Cartan element of this $L_\infty$-algebra is precisely given by a pair $(\mu, r)$, where $\mu \in (s^{-1} \mathfrak{B}')_0 = (\mathfrak{B}')_1 = C^{1 | 0}$ and $r : \mathfrak{h} \rightarrow \mathfrak{g}$ is a linear map subject to satisfy
\begin{align*}
    \llbracket \mu, \mu \rrbracket_\mathsf{B} = 0 \quad \text{ and } \quad \llbracket \llbracket \mu, r \rrbracket_\mathsf{B}, r \rrbracket_\mathsf{B} = 0.
\end{align*}
The condition $\llbracket \mu, \mu \rrbracket_\mathsf{B} = 0$ is equivalent that $\mathfrak{g}$ is a Leibniz algebra and $\mathfrak{h}$ is a representation of it, while $\llbracket \llbracket \mu, r \rrbracket_\mathsf{B}, r \rrbracket_\mathsf{B} = 0$ is equivalent that $r : \mathfrak{h} \rightarrow \mathfrak{g}$ is relative Rota-Baxter operator of weight $0$. Thus, a Maurer-Cartan element precisely describes a relative Rota-Baxter Leibniz algebra structure (i.e. a Leibniz algebra, a representation with a relative Rota-Baxter operator of weight $0$) \cite{das-leib}. Now, given a relative Rota-Baxter Leibniz algebra structure, we can construct the governing $L_\infty$-algebra $\big( s^{-1} \mathfrak{B}' \oplus \mathfrak{a}, \{ \widetilde{l_k}^{(\mu, r)} \}_{k=1}^\infty \big)$ just by the twisting procedure.

Next, let $\mathfrak{B}'' = \oplus_{n=0}^\infty (C^{n+1|-1} \oplus C^{n|0})$. Then it follows from Lemma \ref{lemma-tang} that $\mathfrak{B}'' \subset \mathfrak{B}$ is also a graded Lie subalgebra. Hence the corresponding $L_\infty$-algebra becomes $(s^{-1} \mathfrak{B}'' \oplus \mathfrak{a} , \{ \widetilde{l_k} \}_{k=1}^\infty)$. A Maurer-Cartan element of this $L_\infty$-algebra is given by a pair $(\Omega, r)$, where $\Omega \in (s^{-1} \mathfrak{B}'')_0 = \mathfrak{B}''_1 = C^{2|-1} \oplus C^{1|0}$ (i.e. $\Omega = \widetilde{\theta} + \mu$ with $\widetilde{\theta} \in C^{2|-1}$ and $\mu \in C^{1|0}$ ~\!) and $r: \mathfrak{h} \rightarrow \mathfrak{g}$ is a linear map satisfying
\begin{align}\label{twisted-simul}
    \llbracket \Omega , \Omega \rrbracket_\mathsf{B} = 0 \quad \text{ and } \quad \frac{1}{2!} \llbracket \llbracket \mu, r \rrbracket_\mathsf{B} , r \rrbracket_\mathsf{B} + \frac{1}{3!} \llbracket \llbracket \llbracket \widetilde{\theta}, r \rrbracket_\mathsf{B}, r \rrbracket_\mathsf{B}, r \rrbracket_\mathsf{B} = 0.
\end{align}
The condition $\llbracket \Omega , \Omega \rrbracket_\mathsf{B} = 0$ is equivalent to $\llbracket \mu, \mu \rrbracket_\mathsf{B} = \llbracket \mu, \widetilde{\theta} \rrbracket_\mathsf{B} = 0$ which simply mean that $\mathfrak{g}$ is a Leibniz algebra, $\mathfrak{h}$ is a representation of it and $\theta$ is a Leibniz $2$-cocycle. On the other hand, the second condition of (\ref{twisted-simul}) is equivalent that $r: \mathfrak{h} \rightarrow \mathfrak{g}$ is a $\theta$-twisted Rota-Baxter operator. Thus, a Maurer-Cartan element corresponds to a $\theta$-twisted Rota-Baxter Leibniz algebra structure. Given such a structure (i.e. having a Maurer-Cartan element), we can also obtain the governing $L_\infty$-algebra by the twisting procedure.

It follows from  Proposition \ref{prop-mqr} that $\mathcal{M} = \oplus_{n=0}^\infty \mathcal{M}_n$ is also a graded Lie subalgebra of $\mathfrak{B}$. In this case, the corresponding $L_\infty$-algebra becomes $( s^{-1} \mathcal{M} \oplus \mathfrak{a} , \{ \widetilde{l_k} \}_{k=1}^\infty )$. A Maurer-Cartan element of this $L_\infty$-algebra is precisely a pair $(\Omega , r)$, where $\Omega \in (s^{-1} \mathcal{M})_0 = \mathcal{M}_1 = C^{1|0} \oplus C^{0|1}$ (hence $\Omega = \mu + \nu$ with $\mu \in C^{1|0}$ and $\nu \in C^{0|1}$ ~ \!) and $r : \mathfrak{h} \rightarrow \mathfrak{g}$ is a linear map satisfying
\begin{align}\label{last-matched}
    \llbracket \Omega , \Omega \rrbracket_\mathsf{B} = 0 \quad \text{ and } \quad \llbracket \nu ,  r \rrbracket_\mathsf{B} + \frac{1}{2}\llbracket \llbracket \mu,  r \rrbracket_\mathsf{B} , r \rrbracket_\mathsf{B} = 0.
\end{align}
The condition $\llbracket \Omega, \Omega \rrbracket_\mathsf{B} = 0$ is equivalent that $(\mathcal{G} = \mathfrak{g} \oplus \mathfrak{h}, \Omega)$ is a matched pair of Leibniz algebras, and the second condition of (\ref{last-matched}) is equivalent that $r: \mathfrak{h} \rightarrow \mathfrak{g}$ is a deformation map in the matched pair of Leibniz algebras (see Proposition \ref{def-matched-prop}). Thus, a Maurer-Cartan element describes a matched pair of Leibniz algebras and a deformation map on it. Finally, given such a structure, we can also find the governing $L_\infty$-algebra by the twisting method.

\medskip

\medskip

\noindent {\bf Acknowledgements.} All the authors thank the Department of Mathematics, IIT Kharagpur for providing the beautiful academic atmosphere where the research has been conducted.

\medskip

\noindent {\bf Funding.} Suman Majhi thanks UGC for the PhD research fellowship and Ramkrishna Mandal thanks the Government of India for supporting his work through the Prime Minister Research Fellowship.

\medskip

\noindent {\bf Data Availability Statement.} Data sharing does not apply to this article as no new data were created or analyzed in this study.

\end{document}